\documentclass{llncs}

\usepackage{graphicx}
\usepackage{amssymb}
\usepackage{amsmath}
\usepackage{amssymb,euscript}
\usepackage[english]{babel}
\usepackage[subnum]{cases}

\newtheorem{thm}{Theorem}
\newtheorem{lem}[thm]{Lemma}

\newtheorem{prb}{Problem}

\begin{document}


\title{Positive Numerical Splitting Method for the Hull and White 2D Black-Scholes Equation}

\author{T. Chernogorova\inst{1} \and R. Valkov\inst{1,2} \\ $\lbrace$chernogorova,rvalkov$\rbrace$@fmi.uni-sofia.bg}
\institute{Faculty of Mathematics and Informatics, University of Sofia, 1164 Sofia, Bulgaria, \and Department of Mathematics and Computer Science, \\ University of Antwerp, 2020 Antwerp, Belgium}

\maketitle

\begin{abstract}
We consider the locally one-dimensional backward Euler splitting method to solve numerically the Hull and White problem for pricing European options with stochastic volatility in the presence of a mixed derivative term. We prove the first-order convergence of the time-splitting. The parabolic equation degenerates on the boundary $x=0$ and we apply a fitted finite volume scheme to the equation in order to resolve the degeneracy and derive the fully-discrete problem as we also investigate the discrete maximum principle. Numerical experiments illustrate the efficiency of our difference scheme.

\end{abstract}


\section{Introduction}

There are nowadays many generalizations of the celebrated Black-Scholes model, resulting in linear and nonlinear degenerate backward parabolic problems. Hull and White \cite{HW} proposed a model for valuing an option with stochastic volatility of the price of the underlying asset that constitutes an important two-dimensional extension of the one-dimensional Black-Scholes partial differential equation (PDE) \cite{TR,WHD}. Since no closed-form analytical formulas have been derived for any but the simplistic cases of multidimensional problems in mathematical finance computationally efficient and accurate numerical methods are needed for the general case of time- and path-dependent market parameters. In the years, numerous numerical techniques have been developed \cite{CVBoro,HH,Hout,SWHH1,IT}.

This paper focuses on the numerical solution of the Black-Scholes equation in stochastic volatility models. The features of this parabolic, two-dimensional convection-reaction-diffusion problem is the presence of a \emph{mixed spatial derivative term}, stemming from the correlation between the two underlying stochastic processes for the asset price and its variance, and the \emph{degeneracy} of the parabolic problem on a part of the domain boundary. Well-posedness of degenerate parabolic PDEs such as the Hull and White model does not follow from classical theory and additional analysis is needed \cite{GV,OR}.

Semi-discretization in space of parabolic PDEs by finite difference schemes gives rise to large systems of stiff ODEs. The two-dimensional exponentially fitted finite volume element method, constructed by Huang et al. \cite{SWHH1}, successfully resolves the degeneracy issue of the PDE problem but fails to address the efficient time stepping of the resulting semi-discrete system. Indeed, the standard $\theta$-method is not practical for multidimensional problems and therefore various splitting methods are designed, cf. Hundsdorfer and Verwer \cite{HundVer}. 

We aim at constructing an economical numerical algorithm without compromising the stability by implementing locally one-dimensional splitting backward Euler method (LOD-BE) splitting method in time. The operator splitting and temporal discretization are executed before we handle the degeneracy of the problem in space by using the fitted finite volume method, proposed by Wang \cite{SW} and further developed in \cite{A,CV}. The attractive features of our numerical method are computational efficiency, stability and positivity (short for nonnegativity) of the numerical solution.

In Section \ref{DiffProblem} we formulate the differential problem and present brief analysis of the existence and uniqueness of weak solution in weighted Sobolev spaces as well as the weak maximum principle. Section \ref{LOD} focuses on the construction and analysis of the splitting method. The full discretization is presented in Section \ref{FullDisc} where we also discuss the discrete maximum principle. In Section \ref{NumExp} we analyze experimentally the global error in discrete norms.

\section{The Differential Problem}\label{DiffProblem}

Stochastic volatility models are those where the price of the underlying asset $x$ and its instantaneous variance $y=\sigma^2$, are both considered as random (state) variables, following certain stochastic processes. Hull and White assume that these state variables obey the geometric (exponential) Brownian motion. Therefore, the price $V$ of a European option with stochastic volatility $\sqrt{y}$ and expiry date $T$ by the general PDE for derivatives satisfies the following backward parabolic problem \cite{HW}
\begin{equation}\label{HWPDE}
-\frac{\partial V}{\partial t} - \frac{1}{2} \left[ x^2y \frac{\partial ^2 V}{\partial x^2} + 2 \rho \xi x y^{3/2} \frac{\partial^2 V}{\partial x \partial y} +
\xi^2 y^2 \frac{\partial^2 V}{\partial y^2} \right] - rx \frac{\partial V}{\partial x} - \mu y \frac{\partial V}{\partial y} + rV = 0,
\end{equation}
where $(x,y,t) \in (0,X) \times (\zeta,Y) \times [0,T) =: \Omega \times [0,T)$ with the final (\emph{pay-off}) and Dirichlet boundary conditions on the boundary $\partial \Omega$ of $\Omega$
\begin{eqnarray}
& V(x,y,T)=V_T(x,y), \; (x,y) \in \Omega, \label{TC} \\
& V(x,y,t)=V_D(x,y,t), \; (x,y,t) \in \partial \Omega \times [0,T), \label{BC}
\end{eqnarray}
The parameters $\xi$ and $\mu$ are constants from the stochastic process, governing the variance $y$, $\rho$ is the instantaneous correlation between $x$ and $y$, $\zeta, \; X, \; Y$ and $T$ are positive constants, defining the solution domain. In the following considerations we assume that $V_D(x,y,t)=0$, i.e. we formally subtract some function, satisfying the boundary conditions \eqref{BC}, from both sides of \eqref{HWPDE} so that a non-zero term $g$ is introduced in the right-hand side (r.h.s) of \eqref{HWPDE}.

In this paper we assume that $\rho \in [0,1)$ is a constant, consistent with the considerations in \cite{HW,SWHH1}. It is also reasonable to assume that $y \geq \zeta$ for a (small) positive constant $\zeta$ since the $y=0$ is trivial as the volatility of the stock is zero in the market and therefore the price of the option is deterministic.

Introducing the new variable $u=\exp(\beta t)V$, where $\beta > 0$ is an arbitrary constant, \eqref{HWPDE} is rewritten as the forward parabolic nonhomogeneous equation
\begin{equation}\label{HWPDE2}
\frac{\partial u}{\partial t} - \frac{1}{2} \left[ x^2y \frac{\partial ^2 u}{\partial x^2} + 2 \rho \xi x y^{3/2} \frac{\partial^2 u}{\partial x \partial y} +
\xi^2 y^2 \frac{\partial^2 u}{\partial y^2} \right] - rx \frac{\partial u}{\partial x} - \mu y \frac{\partial u}{\partial y} + (r+\beta)u = g,
\end{equation}
with the homogeneous boundary condition on $\partial \Omega$. Further in our analysis we refer to the forward problem with initial data, corresponding to \eqref{TC}.

\subsection{Well-posedness and maximum principle}

The well-posedness considerations in this subsection are presented by Huang et al. \cite{SWHH2}. We extend their variational analysis by deriving the weak maximum principle for equation \eqref{HWPDE2}, written in the following divergence form
\begin{equation}\label{HWPDE3}
\frac{\partial u}{\partial t}-\nabla \cdot (k(u)) + cu = g,
\end{equation}
$k(u)=A \nabla u + \textbf{b} u$ is the flux, $\textbf{b}=\left(rx-\frac{3}{4}\rho y^{1/2} \xi x - yx, \mu y -\frac{1}{2}\rho \xi y^{3/2} - \xi^2 y\right)^T$,
\begin{eqnarray}\label{divformAc}
\begin{split}
A=\begin{pmatrix}
    a_{11} & a_{12} \\
    a_{21} & a_{22} \\
  \end{pmatrix}
  =\begin{pmatrix}
    \frac{1}{2}yx^2 & \frac{1}{2}\rho y^{3/2}\xi x \\
    \frac{1}{2}\rho y^{3/2}\xi x & \frac{1}{2}\xi^2 y^2 \\
  \end{pmatrix}, \\
  c=\beta + 2r -\frac{3}{4}\rho y^{1/2} \xi - y + \mu - \frac{3}{4} \rho y^{1/2} \xi - \xi^2.
\end{split}
\end{eqnarray}

Let $L^{p}(\Omega)$ denote the space of all $p$-integrable functions on $\Omega$ for $p \geq 1$. For $p=2$ the inner product on $L^{2}(\Omega)$ is given by $(u,v):=\int_{\Omega} uv d\Omega$ with the norm $\|v\|_0^2 :=\int_{\Omega} v^2 d\Omega$. 
To handle the degeneracy in the Hull and White problem the weighted inner product on $\left( L^2(\Omega) \right)^2$ is introduced by $(\textbf{u},\textbf{v})_{\hat{\omega}}:=\int_{\Omega} (yx^2 u_1v_1 + y^2 u_2v_2)d\Omega$ for any $\textbf{u}=(u_1,u_2)^T$ and $\textbf{v}=(v_1,v_2)^T \in \left( L^2(\Omega) \right)^2$. The corresponding weighted $L^2$-norm is
\begin{equation*}
\|\textbf{v}\|_{0,\hat{\omega}}:=\sqrt{(\textbf{v},\textbf{v})_{\hat{\omega}}}=\left( \int_{\Omega} (yx^2v_1^2 +y^2v_2^2)d\Omega \right)^{1/2}.
\end{equation*}
The space of all weighted square-integrable functions is defined as
\begin{equation*}
\textbf{L}_{\hat{\omega}}^2(\Omega):=\left\{ \textbf{v} \in \left( L^2(\Omega) \right)^2:\|v\|_{0,\hat{\omega}} < \infty \right\}.
\end{equation*}
The pair $\left( \textbf{L}_{\hat{\omega}}^2(\Omega),(\cdot,\cdot)_{\hat{\omega}} \right)$ is a Hilbert space (cf., for example, \cite{Kufner}) and further the weighted Sobolev space $H_{\hat{\omega}}^{1}(\Omega)$ is given by
\begin{equation*}
H_{\hat{\omega}}^{1}(\Omega)=\left\{ v:v \in L^2(\Omega), \nabla v \in \textbf{L}_{\hat{\omega}}^2(\Omega) \right\}
\end{equation*}
with the energy norm $\|v\|_{1,\hat{\omega}}^2=|v|_{1,\hat{\omega}}^2 + \|v\|_0^2$ for any $v \in H_{\hat{\omega}}^{1}(\Omega), |v|_{1,\hat{\omega}}^2=\|\nabla v\|_{0,\hat{\omega}}^2$.

The discussion in \cite{SWHH2} reveals that boundary condition at $x=0$ is not needed because of the degeneracy of the equation at this part of boundary, i.e. the solution to Problem \ref{HWPDE} can not take a trace at $x=0$ and this also holds true for the discrete problem. Detailed considerations of this issue can be found in \cite{OR,Zhu}. Nevertheless, when we solve the problem numerically, we may simply choose a particular solution with a homogeneous trace at $x=0$.

The boundary segments of $\Omega$ with $x=X,y=\zeta$ and $y=Y$ are denoted by $\partial \Omega_{D}=\left\{ (x,y) \in \partial \Omega: x \neq 0 \right\}$ so that we introduce
\begin{equation*}
H_{0,\hat{\omega}}^{1}(\Omega)=\left\{ v:v \in H_{\hat{\omega}}^{1}(\Omega) \; \mbox{and} \; v \left|_{\partial \Omega_D} = 0 \right. \right\}.
\end{equation*}
We define the following variational problem, corresponding to \eqref{HWPDE3} and \eqref{TC},\eqref{BC}.
\begin{prb}\label{weakform}
Find $u(t) \in H_{0,\hat{\omega}}^{1}(\Omega)$, satisfying the pay-off \eqref{TC} such that for all $v \in H_{0,\hat{\omega}}^{1}(\Omega)$
\begin{equation*}
\left( \frac{\partial u(t)}{\partial t},v \right) + \textbf{B}(u(t),v;t)=(g,v) \; \mbox{a.e. in} \; (0,T),
\end{equation*}
where
\begin{equation*}
\textbf{B}(u(t),v;t)=(A\nabla u + \underline{b}u,\nabla v)+(cu,v)
\end{equation*}
is a bilinear form and $A, \; \underline{b}$ and $c$ are defined in \eqref{HWPDE3} and \eqref{divformAc}.
\end{prb}

\begin{thm} \cite{SWHH2}
The bilinear form $\textbf{B}(\cdot,\cdot)$ is coercive in $H_{0,\hat{\omega}}^{1}(\Omega)$
\begin{equation*}
\textbf{B}(v,v;t) \geq C \|v\|_{1,\hat{\omega}}^2,
\end{equation*}
where $C$ denotes a positive constant, independent of $v$, and continuous in $H_{0,\hat{\omega}}^{1}(\Omega)$
\begin{equation*}
\textbf{B}(v,w;t) \leq M \|v\|_{1,\hat{\omega}}\|w\|_{1,\hat{\omega}}.
\end{equation*}
There exists an unique solution of Problem \ref{weakform}.
\end{thm}

We are now in position to formulate the following theorem.
\begin{thm}
Let $u(x,y,t) \in H_{0,\hat{\omega}}^{1}(\Omega)$ be the solution of \eqref{HWPDE3},\eqref{TC},\eqref{BC}. If $u_T(x,y) \geq 0$ and $g(x,y,t) \geq 0$ then $u(x,y,t) \geq 0$ a.e. in $Q_T:=\Omega \times (0,T], \; T>0$. 
\end{thm}
\begin{proof}
For a function $u(x,y,t) \in H_{0,\hat{\omega}}^{1}(\Omega)$ we denote the positive and negative parts of $u$ respectively by $u^{+}$ and $u^{-}$, i.e. $u=u^{+}+u^{-}$, $u^{+} \geq 0$ and $u^{-} \leq 0$. Introducing 
\begin{equation*}
Du^+=\left\{
       \begin{array}{ll}
         Du, & \; \mbox{if} \; u > 0, \\
         0, & \; \mbox{if} \; u \leq 0,
       \end{array}
     \right.
Du^-=\left\{
       \begin{array}{ll}
         Du, & \; \mbox{if} \; u < 0, \\
         0, & \; \mbox{if} \; u \geq 0,
       \end{array}
     \right.
\end{equation*}
where $D$ denotes derivative in classical sense we get for any indices $i,j$ (cf. Section 7.4 in Gilbarg and Trudinger \cite{GT})
\begin{equation*}
u^+u^-=D_iu^+D_ju^-=D_iu^+u^-=u^+D_iu^-=0 \;\; \mbox{a.e. in} \;\; \Omega.
\end{equation*}
Further, we consider the strong variational form of \eqref{HWPDE2} in $Q_t$
\begin{equation*}
\int_{Q_t}\int \left( \frac{\partial u}{\partial t}-\nabla \cdot (k(u)) + cu \right) v d\Omega dt = \int_{Q_t}\int gv d\Omega dt.
\end{equation*}
Therefore we have
\begin{eqnarray}\label{mpeq1}
\begin{split}
& \int_{\Omega} uv d\Omega - \int_{\Omega} u(x,0)v(x,0) d\Omega - \int_{Q_T}\int u\frac{\partial v}{\partial t} d\Omega dt - \int_{Q_T}\int gv d\Omega dt \\
& + \int_{Q_T}\int (A\nabla u + \underline{b}u) \cdot \nabla v + cuv d \Omega dt=\int_{0}^{t}\int_{\partial \Omega} (A\nabla u + \underline{b}u) v \cdot \vec{\textbf{n}} d \sigma dt.
\end{split}
\end{eqnarray}
Using Steklov average and passing to the limit \cite{LS} we formally take $v=-u^{-} \geq 0$ in \eqref{mpeq1} to obtain
\begin{eqnarray*}
\begin{split}
-\frac{1}{2}\int_{\Omega} (u^-(x,t))^2 d \Omega + \frac{1}{2}\int_{\Omega} (u^-(x,0))^2 d \Omega - \int_0^t B(u^-,u^-;t) dt \\ 
= \int_{0}^{t}\int_{\partial \Omega} (A\nabla u + \underline{b}u) u^- \cdot \vec{\textbf{n}} d s dt - \int_{Q_T}\int gu^- d\Omega dt.
\end{split}
\end{eqnarray*}
Since $u_T(x,y) \geq 0$, $g(x,y,t) \geq 0$ we have $u^-(x,y,0)=u^-(x,y,t)\left|_{\partial \Omega} \right. \equiv 0$ and
\begin{equation*}
-\frac{1}{2}\int_{\Omega} (u^-(x,t))^2 d \Omega - \int_0^t B(u^-,u^-;t) dt = - \int_{Q_T}\int gu^- d\Omega dt \geq 0.
\end{equation*}
Following the coercivity of the bilinear form $\textbf{B}(\cdot,\cdot;t)$ we arrive at
\begin{equation}\label{mpeq2}
\frac{1}{2}\int_{\Omega} (u^-(x,t))^2 d \Omega + C \int_0^t \|u^{-}\|_{1,\hat{\omega}}^2 dt \leq 0.
\end{equation}
Finally, \eqref{mpeq2} implies $\int_{\Omega} (u^-(x,t))^2 d \Omega$ and therefore $u^-(x,y,t) \equiv 0$. We conclude that $u(x,y,t) \geq 0$ for a.e. $t \in (0,T)$. $\Box$
\end{proof}


\subsection{Pay-off and boundary conditions}

We now discuss in details the pay-off and boundary conditions \eqref{TC},\eqref{BC}, which are, in general, determined by the nature of the option (only call option is considered for brevity). Three typical choices of pay-off functions are considered \cite{WHD} and they are independent of $y$.

The ramp pay-off, corresponding to the vanilla option, is given by
\begin{equation}\label{FinalC1}
u_T(x,y)=\max(0,x-E), \; (x,y) \in \bar{I}_x \times \bar{I}_y,
\end{equation}
where $E<X$ denotes the exercise price of the option, $I_x=(0,X)$ and $I_y=(0,Y)$. The second choice is the cash-or-nothing (digital) pay-off, given by
\begin{equation}\label{FinalC2}
u_T(x,y)=B H(0,x-E), \; (x,y) \in \bar{I}_x \times \bar{I}_y,
\end{equation}
where $B>0$ is a constant and $H$ denotes the Heaviside function. The bullish vertical spread pay-off is defined by
\begin{equation}\label{FinalC3}
u_T(x,y)=\max(0,x-E_1)-\max(0,x-E_2), \; (x,y) \in \bar{I}_x \times \bar{I}_y,
\end{equation}
where $E_1$ and $E_2$ are two exercise prices, satisfying $E_1<E_2$. This represents a portfolio of buying one call option with exercise price $E_1$ and issuing one call option with the same expiry date but a larger exercise price, $E_2$. 

The boundary conditions at $x=0$ and $x=X$ are simply taken to be the extension of the pay-off condition, i.e.
\begin{equation}\label{BCondx}
u_D(0,y,t)=u_T(0,y)=0 \; \mbox{and} \; u_D(X,y,t)=u_T(X,y).
\end{equation}
The boundary condition at $y=\zeta$ ($y=Y$) is the numerical solution of the standard one-dimensional Black-Scholes equation for $\xi=\mu=0$ and the particular value $\sigma=\sqrt{\zeta}$ ($\sigma=\sqrt{Y}$), computed by the algorithm in \cite{SW}.

\section{The splitting method}\label{LOD}


Dimensional time-splitting methods are known for their efficiency, \emph{decreasing significantly the computational costs} when solving multidimensional problems \cite{HundVer,Sam}. Splitting schemes are, in general, categorized in locally one-dimensional (LOD) \cite{Yanenko} methods and alternating directions implicit (ADI) \cite{PR} methods. Predictor-corrector schemes of Douglas type are also considered as ADI schemes although the corrector steps are effectively one-dimensional. The construction of ADI schemes can be regarded as discretization of the PDE, factorization of the discrete equation and splitting of the factored discrete equation. The LOD methods, on the other hand, are based on the method of fractional steps, i.e. splitting of the differential equation, and further discretization of the resulting one-dimensional equations. 

\subsection{Splitting the differential equation}

We follow the method of fractional steps and construct the LOD-BE semi-discrete scheme that can be interpreted as dimensional Rothe method. The equation \eqref{HWPDE2} is further rewritten in the following conservative form
\begin{eqnarray}\label{divPDE}
\begin{split}
& \frac{\partial u}{\partial t} \underbrace{- \frac{\partial}{\partial x}\left(a_{11}\frac{\partial u}{\partial x} + \left( b_1 - \frac{\partial a_{12}}{\partial y} \right)u \right) + c_1 u}_{\mathcal{L}_1 u} \\
& \underbrace{- \frac{\partial}{\partial y}\left(a_{22}\frac{\partial u}{\partial y} + \left( b_2 + \frac{\partial a_{21}}{\partial x} \right)u \right) + c_2 u}_{\mathcal{L}_2 u} - \underbrace{ \frac{\partial}{\partial y}\left((a_{12}+a_{21})\frac{\partial u}{\partial x} \right)}_{\mathcal{L}_3 u}=g_1 + g_2,
\end{split}
\end{eqnarray}
where $a_{11}$, $a_{22}$, $a_{12}=a_{21}$ and $b_1$, $b_2$ are as given in \eqref{divformAc}, $c_1+c_2=c$ and $g_1+g_2=g$. Our flux-based finite volume spatial discretization benefits from the following representation for $k(x,y)=\rho \xi x y^{3/2}$
\begin{equation*}
\frac{\partial u}{\partial t} - \frac{\partial}{\partial x}\left(x w_1(x,y,u) \right) + c_1u
- \frac{\partial}{\partial y}\left(y w_2(y,u) \right) + c_2u - \frac{\partial}{\partial y}\left( k(x,y)\frac{\partial u}{\partial x} \right) = g,
\end{equation*}
where we introduced the weighted flux notation in direction $x$
\begin{equation}\label{fluxx}
w_1(x,y,u)=\frac{1}{2}xy \frac{\partial u}{\partial x} + \left(r-y-\frac{3}{2}\rho \xi y^{1/2}\right) u =: p_1(y)x \frac{\partial u}{\partial x} + q_1(y)u
\end{equation}
and the weighted flux in direction $y$
\begin{equation}\label{fluxy}
w_2(y,u)=\frac{1}{2}\xi^2 y \frac{\partial u}{\partial y} + \left( \mu-\xi^2 \right)u=:p_2y\frac{\partial u}{\partial y} + q_2u.
\end{equation}
For the reaction terms we have 
\begin{equation}
c_1(y)=\frac{3}{2}r-y-\frac{3}{2}\rho \xi y^{1/2}+\frac{1}{2}\beta, \;\; c_2=\frac{1}{2}r + \mu-\xi^2+\frac{1}{2}\beta .
\end{equation}

We introduce a uniform partition of $[0,T] \left\{ {t_k = k\tau ,k = 0,1,\ldots,K,\tau = \frac{T}{K}} \right\}$ and consider the following fractional steps scheme
\begin{eqnarray*}
u_{(1)} \left\{
\begin{array}{llll}
& \frac{\partial u_{(1)} }{\partial t} + \mathcal{L}_1 u_{(1)}=g_1,\quad t_k < t \leq t_{k+1}, \\
& u_{(1)} (x,y,0) = u_T(x), \; (x,y) \in [0,X] \times [\zeta,Y], \\
& u_{(1)} (0,y,t) = u_D(0,y,t), \; (y,t) \in [\zeta,Y] \times (0,T], \\
& u_{(1)} (X,y,t) = u_D(X,y,t), \; (y,t) \in [\zeta,Y] \times (0,T],
\end{array}
\right.
\end{eqnarray*}
\begin{eqnarray*}
u_{(2)} \left\{
\begin{array}{llll}
& \frac{\partial u_{(2)} }{\partial t} + \mathcal{L}_2 u_{(2)}=\mathcal{L}_3 u_{(1)} + g_2,\quad t_{k} < t \leq t_{k+1},k = 1,2,\ldots,K, \\
& u_{(2)} (x,y,t_{k}) = u_{(1)} (x,y,t_{k + 1}), \; (x,y) \in [0,X] \times [\zeta,Y], \\
& u_{(2)} (x,\zeta,t) = u_D(x,\zeta,t), \; (x,t) \in [0,X] \times (0,T], \\
& u_{(2)} (x,Y,t) = u_D(x,Y,t), \; (x,t) \in [0,X] \times (0,T].
\end{array}
\right.
\end{eqnarray*}

The construction of our numerical scheme demands that \emph{we begin by performing the time discretization}. We obtain semi-discrete approximations $u^{k}(x,y)$ to the solution $u(x,y,t)$ of \eqref{HWPDE}-\eqref{BC} at $t=t_k=k\tau$ by the backward Euler time stepping
\begin{eqnarray}
& (I + \tau \mathcal{L}_1)u^{k+1/2} = u^{k} + \tau g_1, \label{FirstSubProb}     \\
& u^0=u_T(x,y), \label{FirstSubProbIC}\\
& u^{k+1/2}(0,y)=u_D(0,y,t_{k+1}), \; u^{k+1/2}(X,y)=u_D(X,y,t_{k+1}), \label{FirstSubProbBC}   \\
& (I + \tau \mathcal{L}_2)u^{k+1} = (I+\tau \mathcal{L}_3)u^{k+1/2} + \tau g_2, \label{SecondSubProb}    \\
& u^{k+1}(x,\zeta)=u_D(x,\zeta,t_{k+1}), \; u^{k+1}(x,Y)=u_D(x,Y,t_{k+1}).\label{SecondSubProbBC}
\end{eqnarray}
We stress on the time stepping in \eqref{SecondSubProb} that may be regarded as implicit-explicit (IMEX) with respect to the intermediate semi-discrete solution $u^{k+1/2}(x,y)$.

\subsection{Analysis of time semi-discretization}

Prior to the next considerations we have to introduce the following weighted Sobolev space 
\begin{equation*}
H_{w}^{1}(0,X)=\left\{ v:v \in L^2(0,X), \nabla v \in L_{w}^2(0,X) \right\},
\end{equation*}
taking into account the degeneracy of the one-dimensional Black-Scholes equation at $x=0$ by the weighted $L^2$-norm
\begin{equation*}
\|v\|_{0,w}:=\sqrt{(v,v)_{w}}=\left( \int_{0}^X x^2 v^2 dx \right)^{1/2},
\end{equation*}
with the energy norm, defined by $\|v\|_{1,w}^2=|v|_{1,w}^2 + \|v\|_0^2$ for any $v \in H_{w}^{1}(0,X)$, where $|v|_{1,w}^2=\|\nabla v\|_{0,w}^2$. We also introduce the following subspace of $H_{w}^{1}(0,X)$
\begin{equation*}
H_{0,w}^{1}(0,X)=\left\{ v:v \in H_{w}^{1}(0,X) \; \mbox{and} \; v(0)=v(X)= 0 \right\}.
\end{equation*}

\begin{lem} Let the operator $(I+\tau \mathcal{L}_1)^{-1}$ be such that $(I+\tau \mathcal{L}_1)^{-1}u$ is the solution $v$ of
\[(I+\tau \mathcal{L}_1)v=u, v(0,y)=0, v(X,y)=0
\]
and analogously for $(I+\tau \mathcal{L}_2)^{-1}$. Then $I+\tau \mathcal{L}_1$ and $I+\tau \mathcal{L}_2$ are inverse positive and satisfy the conditions
\begin{equation}\label{eq:NormEst}
\left\|(I+\tau \mathcal{L}_1)^{-1}\right\|_{L^2(\Omega)}\leq \frac{1}{1+\tilde{C}_1\tau}, \left\|(I+\tau \mathcal{L}_2)^{-1}\right\|_{L^2(\Omega)}\leq \frac{1}{1+\tilde{C}_2\tau},
\end{equation}
where $\tilde{C}_1$, $\tilde{C}_2$ are positive constants independent of $\tau$. 

Moreover, the following estimate holds
\begin{equation}
\left\|(I+\tau \mathcal{L}_2)^{-1}(I+\tau \mathcal{L}_3)(I+\tau \mathcal{L}_1)^{-1}\right\|_{L^2(\Omega)}\leq \frac{1}{1+2\sqrt{\tilde{C}_1 \tilde{C}_2}\tau}.
\label{eq:Est2}
\end{equation}
\end{lem}
\begin{proof} Let $\mathcal{A}:=I+\tau \mathcal{L}_1$. Let us recall the notation for the flux \eqref{fluxx}, suppressing the dependence of the coefficients on $y$. Applying integration by parts one obtains
\begin{equation}\label{eq:Formula23}
	\begin{aligned}
		\left(\mathcal{A}u^{k+1/2},u^{k+1/2}\right)_{L^2(\Omega)}&=\tau\int_{\Omega}x\left(p_1 x\frac{\partial u^{k+1/2}}{\partial x}+q_1 u^{k+1/2}\right)\frac{\partial u^{k+1/2}}{\partial x} dxdy\\
		&+(1+\tau c_1)\left\|u^{k+1/2}\right\|_{L^2(\Omega)}=\int_{\Omega}u^{k}u^{k+1/2} dxdy,
	\end{aligned}
\end{equation}
and then 
\begin{equation}\label{eq:Estimate1}
	\begin{aligned}	
		\left(1+\tau \left(c_1-\frac{1}{2}q_1\right)\right)\left\|u^{k+1/2}\right\|^2_{L^2(\Omega)}&\leq\frac{1}{2}\tau\int_{\Omega}x^2y\left(\frac{\partial u^{k+1/2}}{\partial x}\right)^2 dxdy\\
		&+\left(1+\tau \left(c_1-\frac{1}{2}q_1\right)\right)\left\|u^{k+1/2}\right\|^2_{L^2(\Omega)}\\
		&\leq \left\|u^{k}\right\|_{L^2(\Omega)}\left\|u^{k+1/2}\right\|_{L^2(\Omega)}.
	\end{aligned}
\end{equation}
Choose $\beta$ such that 
\[c_1-\frac{1}{2}q_1=r-\frac12y-\frac34\rho\xi y^{1/2}+\frac12\beta\geq \tilde{C}_1>0.\] Then 
the first estimate in \eqref{eq:NormEst} holds and the second is derived analogously. 

Denote $\mathcal{B}=I+\tau \mathcal{L}_2$, and let $u^{k+1}$ be the solution of $\mathcal{B}u^{k+1}=(I+\tau \mathcal{L}_3)u^{k+1/2}$ where $u^{k+1/2}$ is the corresponding solution of $\mathcal{A}u^{k+1/2}=(I+\tau \mathcal{L}_1)u^{k+1/2}=u^k$. Obviously,
\[u^{k+1}=(I+\tau \mathcal{L}_2)^{-1}(I+\tau \mathcal{L}_3)(I+\tau \mathcal{L}_1)^{-1}u^k.
\]
Again, integrating by parts, we have
\begin{eqnarray*}
\begin{split}
\left(\mathcal{B}u^{k+1},u^{k+1}\right)_{L^2(\Omega)}&=\frac{1}{2}\tau\xi^2\int_{\Omega}y^2\left(\frac{\partial u^{k+1}}{\partial y}\right)^2 dxdy \\
&+\left(1+\tau \left(c_2-\frac{1}{2}q_2\right)\right)\left\|u^{k+1}\right\|^2_{L^2(\Omega)}
\end{split}
\end{eqnarray*}
and also
\begin{eqnarray*}
\begin{split}
\left((I+\tau \mathcal{L}_3)u^{k+1/2},u^{k+1}\right)_{L^2(\Omega)}&=\left(u^{k+1/2},u^{k+1}\right)_{L^2(\Omega)} \\
&-\tau\rho\xi\int_{\Omega}xy^{3/2}\frac{\partial u^{k+1}}{\partial y} \frac{\partial u^{k+1/2}}{\partial x}dxdy.
\end{split}
\end{eqnarray*}
Then the right-hand sides coincide implying that
\begin{eqnarray}\label{eq:Estimate2}
\begin{split}
&\frac{1}{2}\tau\int_{\Omega}\left(y\xi\frac{\partial u^{k+1}}{\partial y}+\rho xy^{1/2} \frac{\partial u^{k+1/2}}{\partial x}\right)^2 dxdy \\
&+\left(1+\tau \left(c_2-\frac{1}{2}q_2\right)\right)\left\|u^{k+1}\right\|^2_{L^2(\Omega)}\\
&\qquad=\left(u^{k+1/2},u^{k+1}\right)_{L^2(\Omega)}+\frac{1}{2}\tau\rho^2\int_{\Omega}x^{2}y\left(\frac{\partial u^{k+1/2}}{\partial x}\right)^2dxdy.
\end{split}
\end{eqnarray}
We use the second inequality in \eqref{eq:Estimate1} and obtain
\begin{equation}
\begin{aligned}
&\left(1+\tau \tilde{C}_2\right)\left\|u^{k+1}\right\|^2_{L^2(\Omega)}\leq\left\|u^{k+1/2}\right\|_{L^2(\Omega)}\left\|u^{k+1}\right\|_{L^2(\Omega)}\\
&\quad+\rho^2\left(\left\|u^{k+1/2}\right\|_{L^2(\Omega)}\left\|u^{k}\right\|_{L^2(\Omega)}-\left(1+\tau \tilde{C}_1\right)\left\|u^{k+1/2}\right\|^2_{L^2(\Omega)}\right)\\
&\quad\leq\left\|u^{k+1/2}\right\|_{L^2(\Omega)}\left(\left\|u^{k+1}\right\|_{L^2(\Omega)}+\left\|u^{k}\right\|_{L^2(\Omega)}-\left(1+\tau \tilde{C}_1\right)\left\|u^{k+1/2}\right\|_{L^2(\Omega)}\right).
\end{aligned}
\label{eq:Estimate3}
\end{equation}
Since the last expression is quadratic with respect $\left\|u^{k+1/2}\right\|_{L^2(\Omega)}$ we get 
\[\left(1+\tau \tilde{C}_2\right)\left\|u^{k+1}\right\|^2_{L^2(\Omega)}\leq \frac{\left(\left\|u^{k+1}\right\|_{L^2(\Omega)}+\left\|u^{k}\right\|_{L^2(\Omega)}\right)^2}{4\left(1+\tau \tilde{C}_1\right)}
\]
and further we obtain
\[4\left(1+\tau \tilde{C}_1\right)\left(1+\tau \tilde{C}_2\right)\left\|u^{k+1}\right\|^2_{L^2(\Omega)}\leq \left(\left\|u^{k+1}\right\|_{L^2(\Omega)}+\left\|u^{k}\right\|_{L^2(\Omega)}\right)^2
\]
which together with $\left(1+\tau \tilde{C}_1\right)\left(1+\tau \tilde{C}_2\right)\geq \left(1+\tau\sqrt{\tilde{C}_1\tilde{C}_2}\right)^2$ implies
\[2\left(1+\tau\sqrt{\tilde{C}_1\tilde{C}_2}\right)\left\|u^{k+1}\right\|_{L^2(\Omega)}\leq\left\|u^{k+1}\right\|_{L^2(\Omega)}+\left\|u^{k}\right\|_{L^2(\Omega)}.
\]
Hence 
\[\left(1+2\tau\sqrt{\tilde{C}_1\tilde{C}_2}\right)\left\|u^{k+1}\right\|_{L^2(\Omega)}\leq\left\|u^{k}\right\|_{L^2(\Omega)}
\]
and then \eqref{eq:Est2} holds. $\Box$
\end{proof}

Further, the consistency of the semi-discretization is investigated. Following the considerations of Clavero et al. \cite{CJL} we define the local error $\phi_{n+1}$ by
\begin{equation*}
\phi_{n+1}=u(x,y,t_{k+1})-\acute{u}^{k+1}(x,y),
\end{equation*}
where $\acute{u}^{k+1}$ is the result '$u^{k+1}$' of applying the semi-discrete scheme with $u^k=u(t_k)$. We then have the next two results, Lemma \ref{leest} and Theorem \ref{tempconv}, under the assumption that the pay-off and the r.h.s. $g$ are sufficiently smooth and compatible for the solution $u$ to have generalized spatial derivatives up to order four and the derivatives w.r.t. $t$ are smooth up to order two.
\begin{lem} \label{leest}
The temporal discretization  \eqref{FirstSubProb}-\eqref{SecondSubProbBC} yields
\begin{equation}\label{lte}
\left\| \phi_{k+1} \right\|_{L^2(\Omega)} \leq C \tau^2,
\end{equation}
where $C$ is a constant, independent of $\tau$.
\end{lem}
\begin{proof} From equations \eqref{FirstSubProb},\eqref{SecondSubProb} we have that $\acute{u}^{k+1}$ satisfies the equation
\begin{eqnarray}\label{acuteu1}
\begin{split}
\acute{u}^{k+1} & = (I+\tau \mathcal{L}_2)^{-1}(I+\tau \mathcal{L}_3)(I+\tau \mathcal{L}_1)^{-1} u(t_k) \\
& + \tau(I+\tau \mathcal{L}_2)^{-1}(I+\tau \mathcal{L}_3)(I+\tau \mathcal{L}_1)^{-1}g_1 + \tau (I+\tau \mathcal{L}_2)^{-1}g_2.
\end{split}
\end{eqnarray}
Further, under some regularity assumptions on $v$, one observes that 
\begin{equation}\label{invop}
(I+\tau \mathcal{L}_1)^{-1}v=(I-\tau \mathcal{L}_1)v + O(\tau^2), \; (I+\tau \mathcal{L}_2)^{-1}v=(I-\tau \mathcal{L}_2)v + O(\tau^2).
\end{equation}
Therefore we derive from \eqref{acuteu1}
\begin{eqnarray*}
\begin{split}
\acute{u}^{k+1} & = (I-\tau \mathcal{L}_2)(I+\tau \mathcal{L}_3)(I-\tau \mathcal{L}_1)u(t_k) \\
& + \tau(I-\tau \mathcal{L}_2)(I+\tau \mathcal{L}_3)(I-\tau \mathcal{L}_1)g_1 + \tau (I-\tau \mathcal{L}_2)g_2 + O(\tau^2).
\end{split}
\end{eqnarray*}
and by executing the multiplication we obtain
\begin{equation}\label{acuteu2}
\begin{split}
\acute{u}^{k+1} & = u(t_k)-\tau(\mathcal{L}_1+\mathcal{L}_2-\mathcal{L}_3)u(t_k) + \tau (g_1+g_2) + O(\tau^2).
\end{split}
\end{equation}
On the other hand, by Taylor series expansion and using \eqref{divPDE}, we get
\begin{eqnarray*}
u(t_{k+1})=u(t_k)-\tau(\mathcal{L}_1+\mathcal{L}_2-\mathcal{L}_3)u(t_k)+\tau(g_1+g_2) + \int_{t_k}^{t^{k+1}}(t_{k+1}-s)\frac{\partial^2 u}{\partial t^2}ds
\end{eqnarray*}
and after subtracting it from \eqref{acuteu2} one derives $\phi_{k+1}=O(\tau^2)$. $\Box$
\end{proof}

We define the global error for the semi-discretization process in the form
\begin{equation*}
\Phi_{\tau}=\sup_{k \leq \frac{T}{\tau}} \left\| u(t_k)-u^{k} \right\|_{L^2(\Omega)}.
\end{equation*}
\begin{thm} \label{tempconv}
The temporal discretization \eqref{FirstSubProb}-\eqref{SecondSubProbBC} is first-order convergent
\begin{equation*}
\Phi_{\tau} \leq C \tau,
\end{equation*}
where $C$ is a constant, independent of $\tau$.
\end{thm}
\begin{proof}
The global error at the time $t_k$ can be decomposed in the form
\begin{equation}\label{errordecomp}
\left\| u(t_k)-u^k \right\|_{L^2(\Omega)} \leq \left\| u(t_k)-\acute{u}^k \right\|_{L^2(\Omega)} + \left\| \acute{u}^k-u^k \right\|_{L^2(\Omega)}.
\end{equation}
For the second term in the right-hand side we have 
\begin{equation*}
\acute{u}^k-u^k = (I+\tau \mathcal{L}_2)^{-1}(I+\tau \mathcal{L}_3)(I+\tau \mathcal{L}_1)^{-1}(u(t_{k-1})-u^{k-1}).
\end{equation*}
Further, by the estimates \eqref{eq:NormEst} and \eqref{eq:Est2}, it holds 
\begin{equation*}
\left\| u(t_k)-u^k \right\|_{L^2(\Omega)} \leq O(\tau^2) + \left\| u(t_{k-1})-u^{k-1} \right\|_{L^2(\Omega)}.
\end{equation*}
Finally, by recurrence we obtain
\begin{equation*}
\left\| u(t_k)-u^k \right\|_{L^2(\Omega)} \leq C \tau. \Box
\end{equation*}
\end{proof}

\section{Full Discretization}\label{FullDisc}

We now proceed to the derivation of the full discretization of problem \eqref{HWPDE}-\eqref{BC}. First, we present the spatial discretization of the equation \eqref{FirstSubProb} in direction $x$ where the main issue is the degeneracy at the boundary $x=0$. Standard centered-space finite difference approximation of the first derivative may introduce oscillations in the numerical solution since the problem is convection-dominated in the neighbourhood where the degeneration occurs.

\subsection{The finite volume scheme}

The exponentially fitted finite volume method of Wang \cite{SW} resolves the degeneracy as the local flux approximation is determined by a set of two-point boundary value problems (BVPs), defined on the element edges. We briefly describe the discussed method as we apply it to the first subproblem \eqref{FirstSubProb}-\eqref{FirstSubProbBC}. Recalling the notations used in the continuous flux \eqref{fluxx} we have
\begin{equation}\label{FirstDivForm}
\begin{split}
\dot{u}^{k+1/2} = \frac{\partial}{\partial x}\left(x\left( p_1(y)x\frac{\partial u^{k+1/2}}{\partial x} + q_1(y)u^{k+1/2} \right) \right) - c_1(y)u^{k+1/2} + g_1^{k+1},
\end{split}
\end{equation}
where $\dot{u}^{k+1/2}$ stands for the discretized temporal derivative. 

Let the interval $[0,X]$ be subdivided into $N$ intervals $I_{i} := (x_{i},x_{i+1}), \; i=0,\dots,N-1,$ with $0 =: x_{0} < x_{1} < \dots < x_{N}:=X$. For each $i=0,\dots,N-1$ we set $h_{i} := x_{i+1} - x_{i}$ and $h := \max_{i=0,\dots,N-1} h_{i}$. We also denote $x_{i+1/2} := x_{i}+h_{i}/2$ for $i=0,\dots,N-1$, $x_{-1/2}:=x_{0}=0$, $x_{N+1/2}:=x_{N}=X$. Analogous partition of $[\zeta,Y]$ is considered in direction $y$. 

Integrating equation \eqref{FirstDivForm} over the interval $\bar{I}_{i}:= [x_{i-1/2},x_{i+1/2}]$ and applying the mid-point quadrature rule we arrive at
\begin{eqnarray}\label{FirstSemiD}
\begin{split}
\dot{u}_i^{k+1/2} \hbar_i^x = x_{i+1/2} \left. w_1(u^{k+1/2}) \right|_{x_{i+1/2}} & - x_{i-1/2} \left. w_1(u^{k+1/2}) \right|_{x_{i-1/2}} \\
& - c_1(y) u^{k+1/2}_i \hbar_i^x + g_1^{k+1} \hbar_i^x,
\end{split}
\end{eqnarray}
where $\hbar_i^x = x_{i+1/2}- x_{i-1/2}$ and $y$ is fixed in these considerations. In order to obtain an approximation for the flux at the node $x_{i+1/2}$ we consider the following boundary value problem (BVP):
\begin{eqnarray*}
\left( {p_{1,{i+1/2}}(y)x{v}' + q_{1,{i+1/2}}(y)v} \right)^\prime = 0, \;\; x \in I_i, \\
v(x_i) = u_i, \; v(x_{i+1}) = u_{i+1}.
\end{eqnarray*}
The solution of that problem is the discrete flux
\begin{equation*}
w_{1,i+1/2}(u)=q_{1,i+1/2}(y)\frac{x_{i+1}^{\alpha_{1,i} (y)} u_{i+1} - x_i^{\alpha_{1,i} (y)} u_i}{x_{i+1}^{\alpha_{1,i} (y)} - x_i^{\alpha_{1,i} (y)}}, \;\;
\alpha_{1,i}(y) = \frac{q_{1,{i+1/2}}}{p_{1,{i+1/2}}}.
\end{equation*}
For the derivation of the discrete flux at $x_{1/2}$, taking into account the degeneracy, the BVP is considered with an extra degree of freedom
\begin{eqnarray*}
\left( {p_{1,1/2}(y)x{v}' + q_{1,1/2}(y)v} \right)^\prime = C, \;\; x \in I_0, \\
v(0) = u_0, \; v(x_1) = u_1,
\end{eqnarray*}
so that we obtain the following approximation:
$$w_{1,1/2}(u)=\frac{1}{2}\left[ \left( p_{1,1/2}(y) + q_{1,1/2}(y) \right)u_1 - \left( p_{1,1/2}(y) - q_{1,1/2}(y) \right)u_0 \right].$$



Dyakonov \cite{Dyakonov} first observed that the boundary conditions deteriorate the accuracy of LOD splitting methods if the discrete equations on the boundaries differ from the equations for the inner nodes of the mesh. The issue is further investigated in \cite{HundVer,Yanenko}. We implement boundary corrections for $j=0$ and $j=M$ so that by the flux approximations we obtain the fully-discrete problem for \eqref{FirstSubProb}
\begin{equation}\label{fulldiscpr1}
\textbf{E}_1 \bar{u}_{j} = f(y_j), \;\; j=0,\dots,N, \;\; \textbf{E}_1=\text{tridiag}\lbrace e_{i,i-1},e_{i,i},e_{i,i+1} \rbrace
\end{equation}
%
where the interior matrix elements for $i=2,\dots,M-1$ are ($p_1(y), \; q_1(y)$ and $\alpha_1(y)$ do not depend on $x$ so we omit the corresponding indexing for clarity)
\begin{eqnarray}\label{interiorelem}
\begin{split}
& e_{i,i-1}=-{\frac{x_{i - 1/2} q_1(y_j)x_{i - 1}^{\alpha_1 (y_j)} }{x_i^{\alpha_1 (y_j)} - x_{i - 1}^{\alpha_1 (y_j)} }}, \;
e_{i,i+1}=-{\frac{x_{i + 1/2} q_1(y_j)x_{i + 1}^{\alpha_1 (y_j)}}{x_{i + 1}^{\alpha_1 (y_j)} - x_i^{\alpha_1 (y_j )} }}, \\
& e_{i,i}=\frac{\hbar _i^x}{\tau} + \frac{x_{i + 1/2}q_1(y_j)x_i^{\alpha_1 (y_j)}}{x_{i + 1}^{\alpha_1 (y_j)} - x_i^{\alpha_1 (y_j)}} + \frac{x_{i - 1/2}q_1(y_j )x_i^{\alpha_1 (y_j )}} {x_i^{\alpha_1 (y_j)}- x_{i-1}^{\alpha_1 (y_j)}} + \hbar_i^x c_1(y_j), \\
\end{split}
\end{eqnarray}
For the discrete equation at $i=1$ we derive
\begin{eqnarray*}\label{modifiedelem}
\begin{split}
& e_{1,0}=-\frac{x_{1/2}}{2}\left(p_1(y_j) - q_1(y_j) \right), \;  
e_{1,2}=-{\frac{x_{3 / 2} q_1(y_j )x_2^{\alpha_1 (y_j )} }{x_2^{\alpha_1 (y_j )} -  x_1^{\alpha_1 (y_j )} }}, \\
& e_{1,1}={\frac{\hbar_1^x}{\tau} + \frac{x_{3/2}p_1(y_j)x_1^{\alpha_1 (y_j )}} {x_2^{\alpha_1 (y_j )} - x_1^{\alpha_1 (y_j )} } + \frac{x_{1 / 2} }{2}\left( {p_1(y_j ) + q_1(y_j )} \right) + \hbar_1^x c_1(y_j)} \\
\end{split}
\end{eqnarray*}
as the right-hand side for $i=1,\dots,M-1$ is
$$
f_{i}(y_j)=\frac{\hbar_i^x}{\tau}u_{i,j}^k+g_1(x_i,y_j,t^{k+1})\hbar_i^x. 
$$
The boundary conditions \eqref{FirstSubProbBC} for $i=0$ and $i=N$ correspond to
\begin{eqnarray*}\label{bcelem}
\begin{split}
& e_{0,0}=1, \; e_{0,1}=0, \; f_{0}(y_j)=u_D(0,y_j,t^{k+1}), \\
& e_{N,N-1}=0, \; e_{N,N}=1, \; f_N(y_j)=u_D(X,y_j,t^{k+1}).
\end{split}
\end{eqnarray*}
%
The discrete problem \eqref{fulldiscpr1} is a linear system for the discrete solution $\bar{u}_{i,j}, \; i=0,\dots,N,$ of the problem \eqref{FirstSubProb}-\eqref{FirstSubProbBC}, solved by the Thomas algorithm.

The spatial discretization of the implicit operator of problem \eqref{SecondSubProb},\eqref{SecondSubProbBC} is derived analogously by the introduction of the continuous flux in direction $y$ \eqref{fluxy}. The mixed derivative term in the r.h.s., however, should be considered in details. For the expression $\left. {\left( {k(x,y)\frac{\partial u}{\partial x}} \right)} \right|_{\left( {x_i ,y_{j - 1/2},t} \right)}^{\left( {x_i ,y_{j+ 1/2},t} \right)} $, $k(x,y) = \rho \xi x y^{3/2}$, we have the following approximation
\begin{eqnarray*}
\begin{split}
& \left. {\left( {k(x,y)\frac{\partial u}{\partial x}} \right)} \right|_{\left( {x_i,y_{j-1/2}} \right)}^{\left( {x_i,y_{j+1/2}} \right)} 
= k_{i,j+1/2} \left. {\frac{\partial u}{\partial x}} \right|_{\left( {x_i ,y_{j+1/2}} \right)} - k_{i,j-1/2} \left. {\frac{\partial u}{\partial x}} \right|_{\left({x_i ,y_{j-1/2}} \right)} \\ &
\approx \frac{k_{i,j+1/2}}{2} \left( {\left.
{\frac{\partial u}{\partial x}} \right|_{\left( {x_i ,y_{j+1}} \right)}+ \left. {\frac{\partial u}{\partial x}} \right|_{\left( {x_i,y_j}
\right)} } \right) - \frac{k_{i,j-1/2}}{2} \left({\left. {\frac{\partial u}{\partial x}} \right|_{\left( {x_i,y_j} \right)} + \left. {\frac{\partial u}{\partial x}} \right|_{\left( {x_i,y_{j-1}}\right)} } \right) \\ 
& \approx \frac{k_{i,j+1/2}}{4} \left( {\frac{u_{i + 1,j + 1} - u_{i,j + 1} + u_{i + 1,j} - u_{i,j} }{h_i^x } + \frac{u_{i,j + 1} - u_{i - 1,j + 1} +
u_{i,j} - u_{i - 1,j} }{h_{i - 1}^x }} \right) \\
& - \frac{k_{i,j-1/2}}{4}\left( \frac{u_{i + 1,j} - u_{i,j} + u_{i + 1,j - 1} - u_{i,j - 1} }{h_i^x } + \frac{u_{i,j} - u_{i - 1,j} + u_{i,j - 1} - u_{i - 1,j - 1} }{h_{i - 1}^x } \right).
\end{split}
\end{eqnarray*}
We stress that at $i=0$ and $i=N$ (the nodes where boundary corrections are applied) the discretization of the mixed derivative is changed respectively by forward and backward difference approximations in direction $x$ as follows
\begin{eqnarray*}
\begin{split}
 \left. {\left( {k(x,y)\frac{\partial u}{\partial x}} \right)} \right|_{\left( {0,y_{j-1/2}} \right)}^{\left( {0,y_{j+1/2}} \right)}
& \approx \frac{k_{0,j+1/2}}{2} \left( \frac{u_{1,j + 1} - u_{0,j + 1} + u_{1,j} - u_{0,j} }{h_0^x } \right) \\
& - \frac{k_{0,j-1/2}}{2} \left( \frac{u_{1,j} - u_{0,j} + u_{1,j - 1} - u_{0,j - 1} }{h_0^x } \right), \\
\end{split}
\end{eqnarray*}
\begin{eqnarray*}
\begin{split}
\left. {\left( {k(x,y)\frac{\partial u}{\partial x}} \right)} \right|_{\left( {X,y_{j-1/2}} \right)}^{\left( {X,y_{j+1/2}} \right)}
& \approx \frac{k_{N,j+1/2}}{2} \left( \frac{u_{N,j + 1} - u_{N-1,j + 1} + u_{N,j} - u_{N-1,j} }{h_{N-1}^x } \right) \\
& - \frac{k_{N,j-1/2}}{2} \left( \frac{u_{N,j} - u_{N-1,j} + u_{N,j - 1} - u_{N-1,j - 1} }{h_{N-1}^x } \right). \\
\end{split}
\end{eqnarray*}

The resulting linear system for the discrete solution $\hat{u}_{i,j}, \; j=0,\dots,N,$ of the problem \eqref{SecondSubProb},\eqref{SecondSubProbBC} is solved by the Thomas algorithm for each $i=0,\dots,N$.

\subsection{Discrete Maximum (Minimum) Principle}

In this subsection we discuss the solvability of the fully-discrete problems as we investigate the monotonicity of the system matrices $\textbf{E}_1$ and $\textbf{E}_2$, corresponding, respectively, to the intermediate solution $\bar{u}$ and the numerical solution on the new time level $\hat{u}$, and the discrete maximum principle. The construction of numerical schemes on compact stencils, obeying the discrete local maximum principle, for two(and higher)-dimensional problems with mixed derivatives is a considerable challenge. Not much has been done in this direction even for the discrete elliptic (or fully implicit) problem but aiming for a generalized local maximum principle \cite{Brandt} or the global maximum principle \cite{BH}. 

In the computational finance literature the issue is investigated by, e.g. Zvan et al. \cite{Zvan}, whereas Ikonen and Toivanen manage to construct a monotone (of positive type) scheme on compact stencil with fully implicit time stepping for the Heston model but under rather restrictive conditions on the mesh \cite{IT}. 

We first present the definitions we further refer to.
\begin{definition} \cite{BH}
A matrix $\textbf{A}$ is said to be \emph{monotone} if $\textbf{A}x \geq 0$ implies $x \geq 0$ for any vector $x$ (to be understood element-wise). 
\end{definition}
\begin{definition} \cite{BH}
An $N \times N$ matrix $\textbf{B}$ with elements $b_{ij}$ is said to be \emph{of positive type} if the following conditions are satisfied:
\begin{itemize}
\item[$\cdot$] $b_{ij} \leq 0, \;\; i \neq j$ (sign condition);
\item[$\cdot$] $\sum_k b_{jk} \geq 0$ for all $j$ with $\sum_k b_{jk} > 0$ for $j \in J(\textbf{B}) \neq 0$ (diagonal dominance, strict for $j \in J(\textbf{B})$);
\item[$\cdot$] for $i \notin J(\textbf{B})$ there exists a finite sequence of non-zero elements of the form $b_{ik_1},b_{k_1k_2},\dots,b_{k_rj}$ where $j \in J(\textbf{B})$ (connection in $\textbf{B}$ from $i$ to $J(\textbf{B})$). 
\end{itemize}
\end{definition}

\begin{theorem}\label{posmatrices}
The (implicit) l.h.s. matrices $\textbf{E}_1,\textbf{E}_2$ are (essentially) of positive type.
\end{theorem}
\begin{proof}
We consider the interior entries of the matrix $\textbf{E}_1$, given by \eqref{interiorelem}. Since
\begin{equation*}
\frac{q_{1,i+1/2}(y_{j})x_{i+1}^{\alpha_1(y_j)}}{x_{i+1}^{\alpha_1(y_j)}-x_{i}^{\alpha_1(y_j)}}=p_{1,i+1/2}(y_{j})\frac{\alpha_1(y_j)}{1-\chi^{\alpha_1(y_j)}}>0, \;\; \chi=\frac{x_{i}}{x_{i+1}},
\end{equation*}
we have $e_{i,i+1}<0$ and analogously $e_{i,i-1}<0$. If the computations are performed for $\beta=0$ then $c_1(y)$ has random sign which results in $\tau=O(1)$ (mild) restriction on the temporal step that guarantees positivity of the diagonal entry $e_{i,i}$ and diagonal dominance. For $\beta>0$ large enough no restriction will be present. Therefore, the 'interior' sub-matrix of $\textbf{E}_1$ for $i=2,\dots,N$ is of positive type which also means it is monotone. 

Further, we investigate the entry $e_{1,0}$, corresponding to the degenerate flux \eqref{modifiedelem}. One may introduce restrictions of the parameters $\xi$ and $\mu$ to ensure the non-positive sign of the entry, e.g. $\mu=0$. We consider the general case of no restrictions on the parameters so that $e_{1,0}>0$. This results in violation of the sign condition and $\textbf{E}_1$ can no longer be considered of positive type. However, \emph{the violation occurs only at the node, adjacent to the boundary}. We can reduce the dimension of the discrete problem by removing $u_{0,j}$ so that we have
$$ 
e_{1,1}u_{1,j}+e_{1,2}u_{2,j}=f_1(y_j)-e_{1,0}u_{0,j}.
$$
Since $e_{1,0}=-\frac{x_{1/2}}{2}\left(p_1(y_j) - q_1(y_j) \right)=O(h_i^x)$ the positivity of the right-hand side is guaranteed for $\tau=O(1)$, regardless whether we consider homogeneous boundary conditions or not. The positivity of $e_{1,1}$ also follows by the same restriction and we have that the reduced matrix $\textbf{E}_1$ is of positive type.
The analysis of the system matrix $\textbf{E}_2$, corresponding to the second sub-problem \eqref{SecondSubProb},\eqref{SecondSubProbBC} is analogous and the reduced system is also of positive type which implies existence and uniqueness of the discrete solution. $\Box$
\end{proof}

Let $\mathcal{L}_h$ is a finite-difference operator. We now recall the definition for discrete operators of positive type
$$
\mathcal{L}_h y := y_{i_1,i_2}^{n+1} - \sum_{(S_{j_1},S_{j_2},t_k) \in \mathcal{S}}^{(i_1,i_2,n+1) \neq (j_1,j_2,l)} b_{j_1,j_2}^l (\tau,h) y_{j_1,j_2}^l, \;\; b_{j_1,j_2}^l \geq 0
$$
on the compact stencil $\mathcal{S}$ that satisfy the following discrete maximum principles.
\begin{definition} \cite{Brandt}
If $\mathcal{L}_h u(x) > 0$ then $u(x)$ is either negative or less than $u(x')$ for some neighbouring grid point $x' \in \mathcal{S}$ and we say that $\mathcal{L}_h$ satisfies the local maximum principle.
\end{definition}
\begin{definition} \cite{Brandt}
If $\mathcal{L}_h u(x) > 0$ for all points $x$ of the discrete domain $\Omega_h$ then it attains its maximum on the boundary of $\Omega_h$ and we say that $\mathcal{L}_h$ satisfies the global maximum principle.
\end{definition} 

Therefore, by Theorem \ref{posmatrices}, the global maximum principle is valid for the discrete problem \eqref{fulldiscpr1} whereas local maximum principle is valid for the reduced problem of the intermediate discrete solution $\bar{u}$ and therefore existence, uniqueness and positivity follow (unconditionally w.r.t. the space discretization step).

Let us now consider the discrete problem for the the numerical solution on the new time level $\hat{u}$. Since $\textbf{E}_2$ is of positive type we have existence and uniqueness of the discrete solution (unconditionally w.r.t. the space discretization step). However, since the applied time discretization is not fully implicit but of IMEX type as the mixed derivative term is treated explicitly w.r.t. $\bar{u}$ the r.h.s. (containing the information from the intermediate solution $\bar{u}$) is nonnegative when $\tau=O(h^2), \; h=\min_{i,j}\lbrace h_i^x,h_j^y \rbrace$ i.e. positivity is conditional.

\section{Numerical Experiments}\label{NumExp}

Numerical experiments, presented in this section, illustrate the properties of the constructed method. We solve numerically various European Test Problems (TP) with different pay-off conditions and different choices of parameters.
\begin{enumerate}
  \item ($TP1$). \emph{Call option} with final condition \eqref{FinalC1}. Parameters: $X=100$, $Y=1$, $T=1$, $\zeta=0.01$, $r=0.1$, $\rho=0.9$, $\xi=1$, $\mu=0$ and $E=57$.
  \item ($TP2$). \emph{Call option} with digital pay-off \eqref{FinalC2}. Parameters: $X=100, \; Y=0.36, \; T=1, \; \zeta=0.01, \; r=0.1, \; \rho=0.9, \; \xi=1, \; \mu=0, \; B=1, \; E=57$.
  \item ($TP3$). \emph{A portfolio of options}. We assume that the final condition is a 'butterfly spread' delta function, defined by
\begin{equation*}
u_T(x,y)=\left\{
       \begin{array}{ll}
         1 , & x \in (X_1,X_2), \\
         -1 , & x \in (X_2,X_3), \\
         0 , & otherwise,
       \end{array}
     \right.
\end{equation*}
and the boundary conditions are assumed to be homogeneous. It arises from a portfolio of three types of options with different exercise prices. Parameters: $X=100$, $Y=0.36$, $T=1$, $X_1=40$, $X_2=50$, $X_3=60$, $\zeta=0.01$, $r=0.1$, $\rho=0.9$, $\xi=1$, $\mu=0$, $B=1$ and $E=57$.
\end{enumerate}

In the tables below are presented the computed $C$ and $L^2$ discrete norms of the error $E=\hat{u}^K-u^K$ by the formulas
\begin{equation*}
\left\|E \right\|_C =\mathop{\max}\limits_{i,j}\left|{\hat{u}_{i,j}^{K} - u_{i,j}^{K}} \right|, \; \left\|E \right\|_{L_2}=\sqrt{\sum\limits_{i = 0}^N {\hbar_i^x \hbar_j^y \left({\hat{u}_{i,j}^{K} - u_{i,j}^{K}}\right)^2}}.
\end{equation*}
We also introduce the root mean square error ($RMSE$) on a specific region
\begin{equation*}
\left\|E \right\|_{RMSE}=\sqrt{\frac{1}{N_{br}}\sum_{i,j}^{br}\left({\hat{u}_{i,j}^{K} - u_{i,j}^{K}}\right)^2},
\end{equation*}
where $N_{br}$ is the number of mesh points in the region we are interested in.
The rate of convergence (RC) is calculated using the double mesh principle
\begin{equation*}
RC=\log_{2}(E^{N,M}/E^{2N,2M}),\;\; E^{N,M}=\|\hat{u}^{N,M}-u^{N,M}\|,
\end{equation*}
where $\|\cdot\|$ is the mesh norm, $u^{N,M}$ and $\hat{u}^{N,M}$ are respectively the exact solution and the numerical solution, computed at the mesh with $N$ and $M$ subintervals in directions $x$ and $y$ respectively.

Table \ref{t1} presents numerical experiments for the exact solution $u=x\exp(-yt)$ with $K=4096$. The choice of this function is motivated by the analytic solution for $\rho=0$, given in \cite{HW}. 
Let us note that when using an exact solution to test the numerical method a r.h.s. arises. The following domain-defining parameters are used: $X=Y=T=1, \; \xi=1$ and $\zeta=0.01$ while the other parameters are selected as $\rho=0.5, \; r=\mu=0$. The results show that the splitting scheme (SplittingFVM) is first-order in space on uniform grid.
\begin{table}[h]
\begin{footnotesize}
\caption{}
\label{t1}
\begin{tabular*}{\textwidth}{@{\extracolsep{\fill}}rcccccccc}
\hline\\[-0.1in]
&&\multicolumn{3}{c}{SplittingFVM}
&&\multicolumn{3}{c}{2DFVM}
\\[0.05in]\cline{3-5}\cline{7-9}\\[-0.1in]
\multicolumn {1}{c}{$N \times M$}
&&\multicolumn {1}{c}{$E^N_\infty$}& $RC$& Norm. CPU
&&\multicolumn {1}{c}{$E^N_\infty$}& $RC$& Norm. CPU
\\[0.03in]
\hline \hline \\[-0.1in]
8x8 && 1.924e-2 & \emph{-} & 1.00 && 1.078e-2 & \emph{-} & 6.62 \\
16x16 && 9.917e-3 & \emph{0.96} & 4.04 && 5.362e-3 & \emph{1.01} & 26.00 \\
32x32 && 4.995e-3 & \emph{0.99} & 16.48 && 2.664e-3 & \emph{1.01} & 105.58 \\
64x64 && 2.502e-3 & \emph{1.00} & 67.89 && 1.327e-3 & \emph{1.01} & 424.03 \\
128x128 && 1.252e-3 & \emph{1.00} & 256.25 && 6.617e-4 & \emph{1.00} & 1776.35 \\
\end{tabular*}
\end{footnotesize}
\end{table}

Comparison of the SplittingFVM with the two-dimensional finite volume method (2DFVM), constructed in \cite{SWHH1}, is also given in Table \ref{t1}. The CPU times are normalized as the time of the splitting scheme on the grid $8 \times 8 \times 4096$ stands for the measure. We observe solid advantage of the splitting method in terms of computational efficiency. The number of the arithmetic operations for computing the numerical solution on the new time level for the SplittingFVM and 2DFVM can be investigated by similar considerations as given in \cite{Sam}.

The exact solution $u(x,y,t)=x\exp(-yt)$ and the corresponding numerical solution, generated by the presented method, are depicted in Figures \ref{exactsol} and \ref{exactnumsol}.
\begin{figure}[htbp]
\hfill%
\begin{minipage}[b]{0.50\textwidth}
\centering
\includegraphics[width=\textwidth]{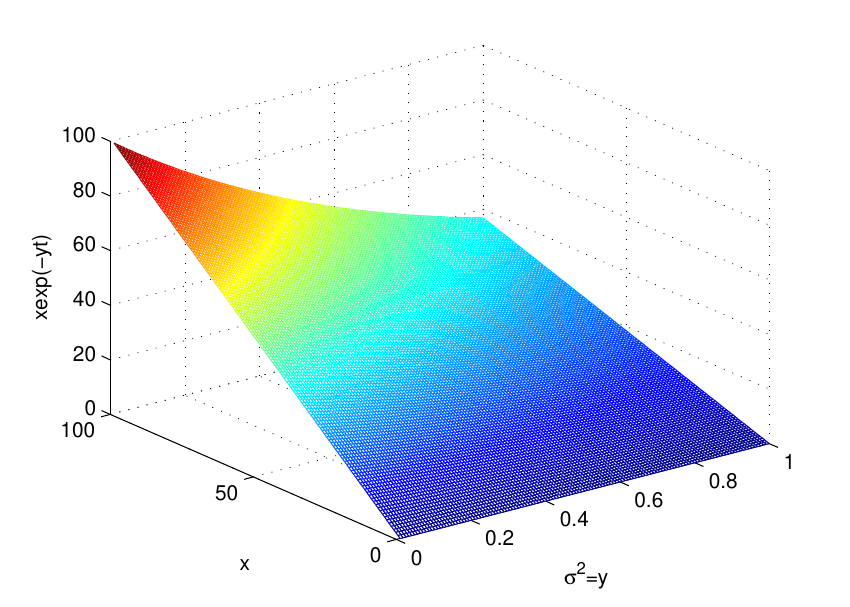}
\caption{exact solution}\label{exactsol}
\end{minipage}%
\hfill%
\begin{minipage}[b]{0.50\textwidth}
\centering
\includegraphics[width=\textwidth]{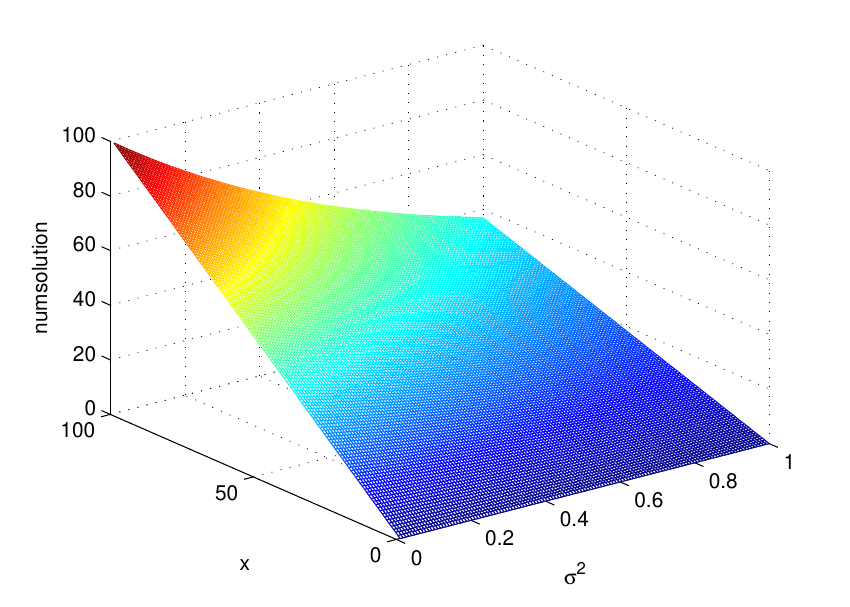}
\caption{numerical solution}\label{exactnumsol}
\end{minipage}\hfill\hbox{}%
\end{figure}

Table \ref{t2} shows the temporal convergence of the numerical solution to the chosen exact solution, $u=x\exp(-yt)$, of \eqref{HWPDE}. We use the same parameters as in Table \ref{t1}: $X=Y=T=1$, $\xi=1$ and $\zeta=0.01$. However, the size of the spatial mesh is now fixed to $512 \times 512$ as the time step varies. The obtained results show that our numerical method profits from the boundary corrections since it is able to sustain the first order of temporal convergence.
\begin{table}[h]
\begin{small}
\caption{}
\label{t2}
\begin{tabular*}{\textwidth}{@{\extracolsep{\fill}}rcccccccccc}
\hline\\[-0.1in]
&&\multicolumn{4}{c}{$\rho=0.5,r=0,\mu=0$}
&&\multicolumn{4}{c}{$\rho=0.9,r=0.1,\mu=0.1$}
\\[0.05in]\cline{3-6}\cline{8-11}\\[-0.1in]
\multicolumn {1}{c}{$K$}
&&\multicolumn {1}{c}{$E^N_\infty$}& $RC$& $E^N_2$& $RC$
&&\multicolumn {1}{c}{$E^N_\infty$}& $RC$& $E^N_2$& $RC$
\\[0.03in]
\hline \hline \\[-0.1in]
16 && 2.000e-2 & \emph{-} & 7.138e-3 & \emph{-} && 3.235e-2 & \emph{-} & 1.157e-2 & \emph{-} \\
32 && 9.859e-3 & \emph{1.02} & 3.585e-3 & \emph{0.99} && 1.562e-2 & \emph{1.05} & 5.753e-3 & \emph{1.01} \\
64 && 4.848e-3 & \emph{1.02} & 1.796e-3 & \emph{1.00} && 7.549e-3 & \emph{1.05} & 2.864e-3 & \emph{1.01} \\
128 && 2.398e-3 & \emph{1.02} & 8.980e-4 & \emph{1.00} && 3.721e-3 & \emph{1.02} & 1.427e-3 & \emph{1.01} \\
256 && 1.197e-3 & \emph{1.00} & 4.477e-4 & \emph{1.00} && 1.862e-3 & \emph{1.00} & 7.099e-4 & \emph{1.01} \\
\end{tabular*}
\end{small}
\end{table}

The equation \eqref{HWPDE} degenerates at $x=0$ and the problem is convection-dominated in this region. One may consider the application of non-uniform grids, analogously to the mesh refinement approach, widely used for singularly perturbed problems \cite{GR}. We present numerical results in Table \ref{t3} with the exact solution $u=x\exp(-yt)$ with $K=1024$ time layers, refining the region of $x=0$,
\begin{equation*}
\eta_i=i\Delta \eta, \; \Delta \eta = \frac{1}{M} \sinh^{-1}(X/d), \; x_i=d \sinh(\eta_i), \; i=0,\dots,N
\end{equation*}
%
\begin{table}[h]
\begin{small}
\caption{}
\label{t3}
\begin{tabular*}{\textwidth}{@{\extracolsep{\fill}}rcccccccccc}
\hline\\[-0.1in]
&&\multicolumn{4}{c}{$h_i^x=d (\sinh(\eta_i)-\sinh(\eta_{i-1}))$}
&&\multicolumn{4}{c}{$h_i^x=X/N$}
\\[0.05in]\cline{3-6}\cline{8-11}\\[-0.1in]
\multicolumn {1}{c}{$N \times M$}
&&\multicolumn {1}{c}{$E^N_\infty$}& $RC$& $E^N_{RMSE}$& $RC$
&&\multicolumn {1}{c}{$E^N_\infty$}& $RC$& $E^N_{RMSE}$& $RC$
\\[0.03in]
\hline \hline \\[-0.1in]
16x128 && 2.9859 & \emph{-} & 0.1301 & \emph{-} && 1.5406 & \emph{-} & 0.7970 & \emph{-} \\
32x128 && 0.9504 & \emph{1.65} & 0.0410 & \emph{1.67} && 0.7703 & \emph{0.99} & 0.3055 & \emph{1.38} \\
64x128 && 0.2640 & \emph{1.85} & 0.0122 & \emph{1.87} && 0.3851 & \emph{1.00} & 0.1159 & \emph{1.40} \\
128x128 && 0.0815 & \emph{1.70} & 0.0029 & \emph{1.95} && 0.1926 & \emph{1.00} & 0.0425 & \emph{1.45} \\
\end{tabular*}
\end{small}
\end{table}
The root mean square error is computed on the region $[0,0.1X] \times [\zeta,Y]$. One observes improvement of the rate of convergence in both norms when using the discussed nonuniform mesh.

We now solve numerically the original problem $TP1$, characterized by non-smoothness of the pay-off \eqref{FinalC1} on an uniform spatial mesh sized $N \times N$ with $2N$ time layers. The boundary conditions (b.c.) in direction $y$ are derived as explained in Section \ref{DiffProblem}, see Figures \ref{bcetatp1},\ref{bcYtp1}. 
In the following Table \ref{t4} the mesh $C$-norm and $RMSE$-norm are computed w.r.t. the numerical solution on a very fine mesh sized $512 \times 512 \times 1024$. The root mean square error is computed on the region $[0.9E,1.1E] \times [\zeta,Y]$ and the numerical solution of $TP1$ is depicted in Figure \ref{optionvaluetp1}.
\begin{figure}[htbp]
\hfill%
\begin{minipage}[b]{0.45\textwidth}
\centering
\includegraphics[width=\textwidth]{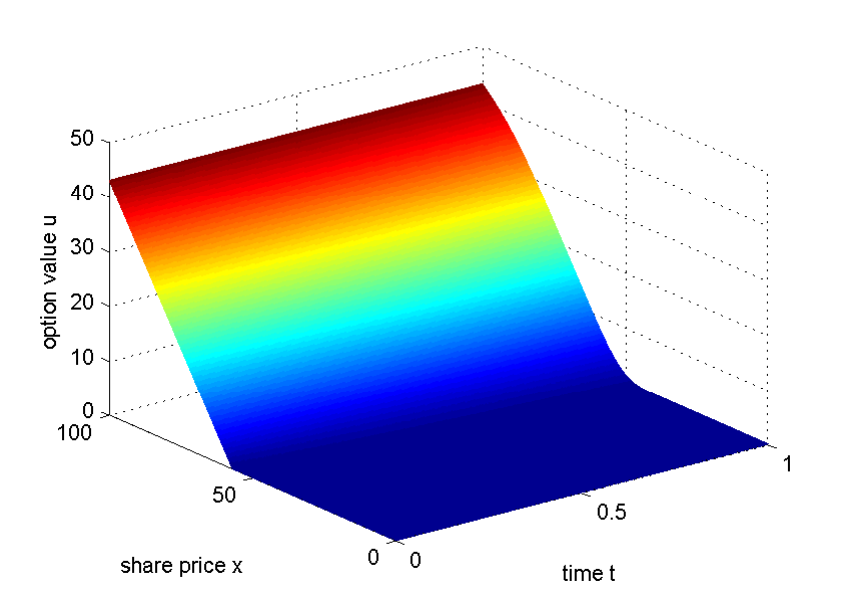}
\caption{b.c. $y=0.01$}\label{bcetatp1}
\end{minipage}%
\hfill%
\begin{minipage}[b]{0.45\textwidth}
\centering
\includegraphics[width=\textwidth]{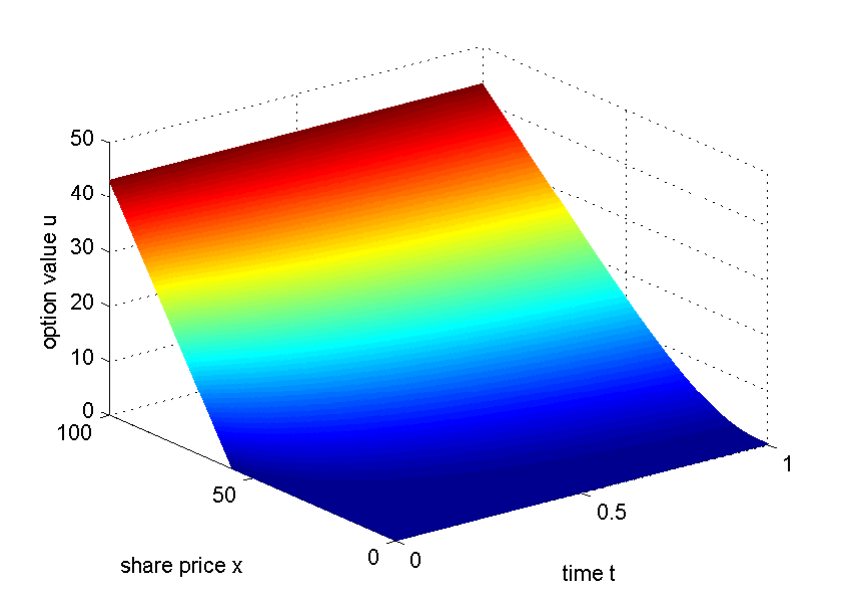}
\caption{b.c. $y=Y$}\label{bcYtp1}
\end{minipage}\hfill\hbox{}%
\end{figure}
\begin{table}
\caption{}
\label{t4}
\centering
\begin{tabular}{cccccc}
\hline
\centering
N & 16 & 32 & 64 & 128 & 256 \\
\hline \hline
$E_{\infty}$ & 2.0678 & 0.9911 & 0.4559 & 0.1944 & 0.0649 \\
        &  & \emph{(1.061)} & \emph{(1.120)} & \emph{(1.230)} & \emph{(1.584)} \\
$E_{RMSE}$ & 0.2649 & 0.1197 & 0.0551 & 0.0236 & 0.0079 \\
        & & \emph{(1.146)} & \emph{(1.119)} & \emph{(1.223)} & \emph{(1.571)} \\
\hline
\end{tabular}
\end{table}
\begin{figure}[htbp]
\hfill%
\begin{minipage}[b]{0.45\textwidth}
\centering
\includegraphics[width=\textwidth]{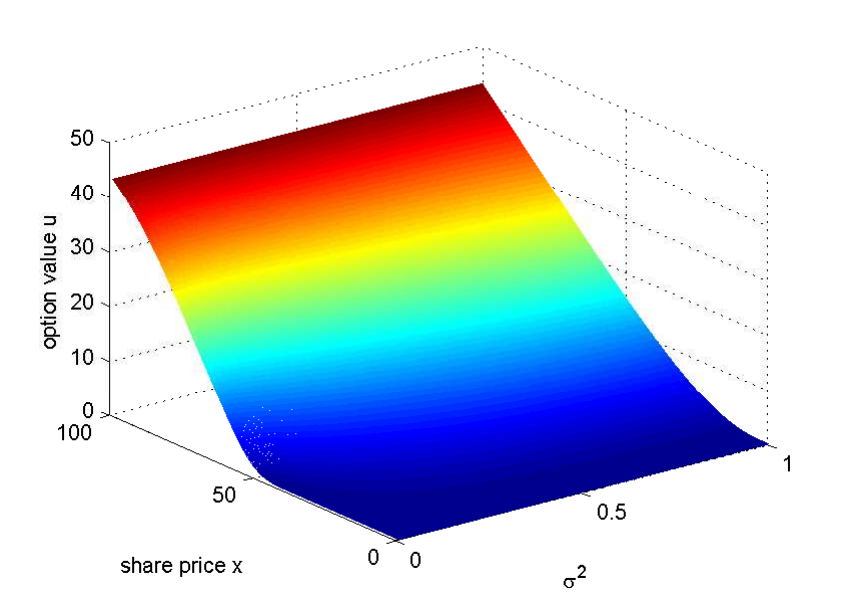}
\caption{option value $TP1$}\label{optionvaluetp1}
\end{minipage}%
\hfill%
\begin{minipage}[b]{0.45\textwidth}
\centering
\includegraphics[width=\textwidth]{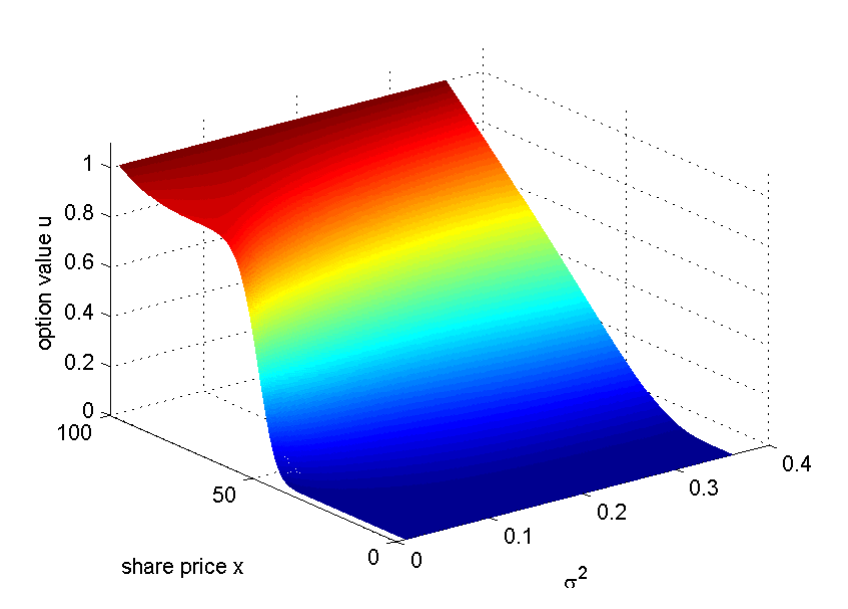}
\caption{option value $TP2$}\label{optionvaluetp2}
\end{minipage}\hfill\hbox{}%
\end{figure}

The discontinuity of the pay-off characterizes the test problems $TP2$ and $TP3$, seriously deteriorating the accuracy. Table \ref{t5} shows results for $TP2$ on a non-uniform mesh, refined in the vicinity of $x=E$, and on an uniform mesh. We use the numerical solution on the fine grid $512 \times 512 \times 1024$ as an exact solution. The mesh size is $N \times N$ with $2N$ time layers and the nodes are generated by the formulas \cite{Hout,TR} with $c=E/5$
\begin{eqnarray*}
& \eta_i=\sinh^{-1}(-E/c) + i\Delta \eta, \; \Delta \eta = \frac{1}{N} \left[ \sinh^{-1}((x-E)/c) - \sinh^{-1}(-E/c) \right], \\
& x_i=E + c \sinh(\eta_i), \; i=0,\dots,N,
\end{eqnarray*}
while the root mean square error is computed on the region $[0.9E,1.1E] \times [\zeta,Y]$. Again, the boundary conditions in direction $y$ are obtained as explained in Section \ref{DiffProblem}, Figures \ref{bcetatp2}, \ref{bcYtp2}, while the numerical solutions for $TP2$ and $TP3$ are shown in Figures \ref{optionvaluetp2}, \ref{optionvaluetp3} respectively.
\begin{table}[h]
\begin{small}
\caption{}
\label{t5}
\begin{tabular*}{\textwidth}{@{\extracolsep{\fill}}rcccccccccc}
\hline\\[-0.1in]
&&\multicolumn{4}{c}{$h_i^x=c (\sinh(\eta_i)-\sinh(\eta_{i-1}))$}
&&\multicolumn{4}{c}{$h_i^x=X/N$}
\\[0.05in]\cline{3-6}\cline{8-11}\\[-0.1in]
\multicolumn {1}{c}{$N$}
&&\multicolumn {1}{c}{$E^N_\infty$}& $RC$& $E^N_{RMSE}$& $RC$
&&\multicolumn {1}{c}{$E^N_\infty$}& $RC$& $E^N_{RMSE}$& $RC$
\\[0.03in]
\hline \hline \\[-0.1in]
32 && 5.953e-2 & \emph{-} & 1.510e-2 & \emph{-} && 8.033e-2 & \emph{-} & 1.903e-2 & \emph{-} \\
64 && 2.642e-2 & \emph{1.17} & 7.040e-3 & \emph{1.10} && 1.442e-2 & \emph{2.48} & 2.412e-3 & \emph{2.98} \\
128 && 1.157e-2 & \emph{1.19} & 3.044e-3 & \emph{1.21} && 1.619e-2 & \emph{-0.17} & 6.630e-3 & \emph{-1.46} \\
256 && 3.802e-3 & \emph{1.61} & 1.018e-3 & \emph{1.58} && 5.497e-3 & \emph{1.56} & 2.204e-3 & \emph{1.59} \\
\end{tabular*}
\end{small}
\end{table}

\begin{figure}[htbp]
\hfill%
\begin{minipage}[b]{0.45\textwidth}
\centering
\includegraphics[width=\textwidth]{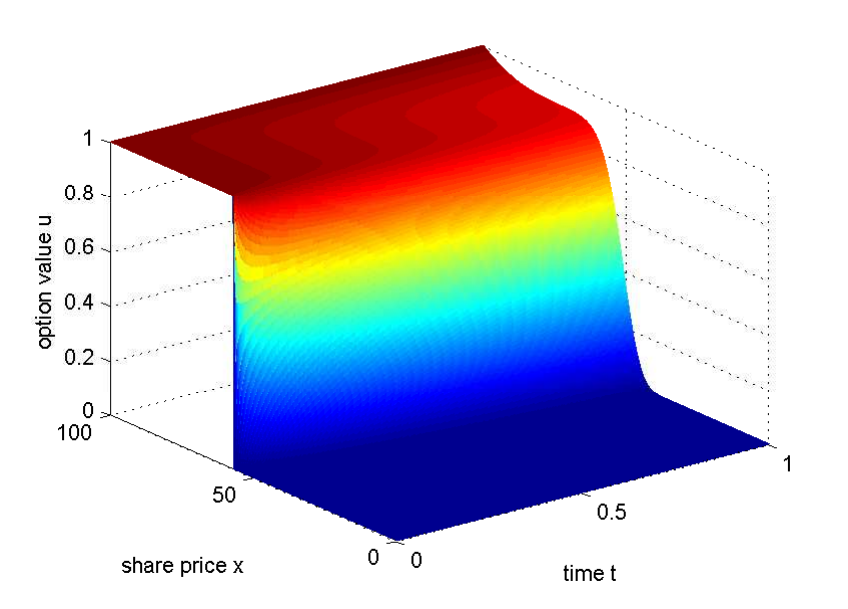}
\caption{b.c. $y=0.01$ $TP2$}\label{bcetatp2}
\end{minipage}%
\hfill%
\begin{minipage}[b]{0.45\textwidth}
\centering
\includegraphics[width=\textwidth]{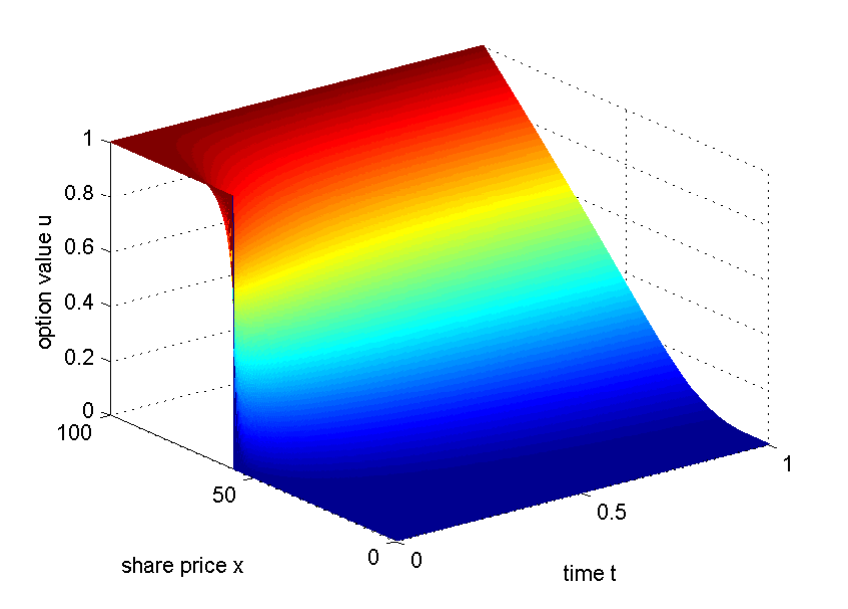}
\caption{b.c. $y=Y$ $TP2$}\label{bcYtp2}
\end{minipage}\hfill\hbox{}%
\end{figure}

\begin{figure}[htbp]
\hfill%
\begin{minipage}[b]{0.45\textwidth}
\centering
\includegraphics[width=\textwidth]{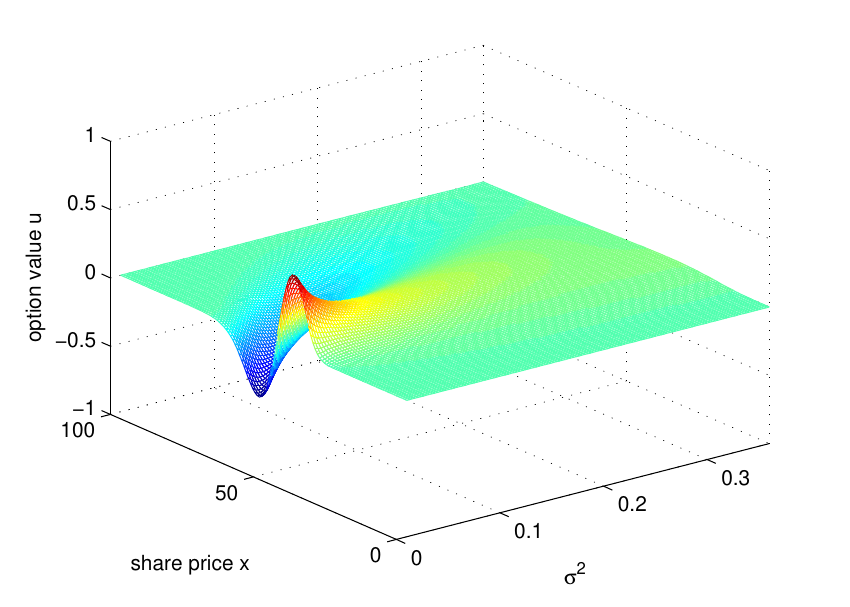}
\caption{option value $TP3$}\label{optionvaluetp3}
\end{minipage}%
\hfill%
\begin{minipage}[b]{0.45\textwidth}
\centering
\includegraphics[width=\textwidth]{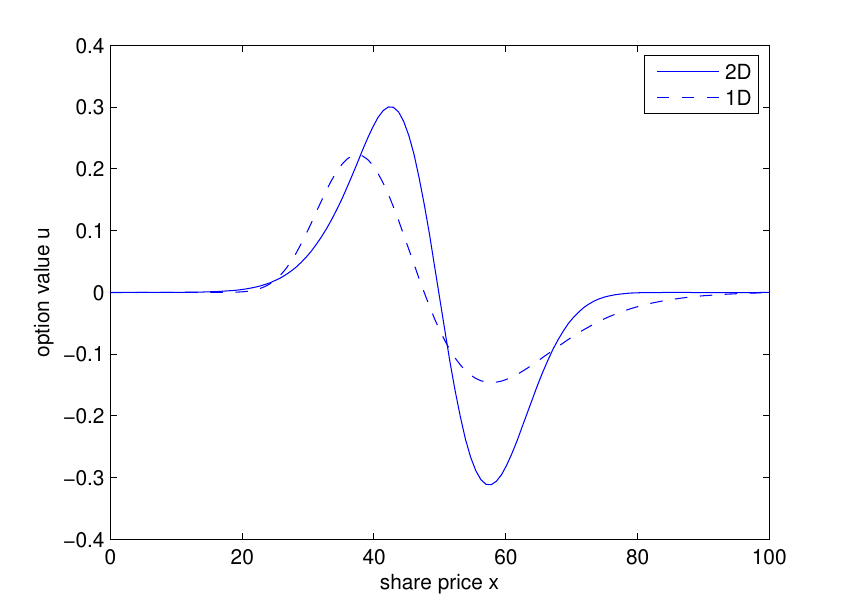}
\caption{2D1D $TP3$ $\sigma \approx 0.20$}\label{comparisontp3}
\end{minipage}\hfill\hbox{}%
\end{figure}

We now present numerical experiments for $\zeta=0$ - a particularly interesting case since one considers degeneration in $y$-direction. The boundary condition on $y=0$ is obtained by taking in consideration the deterministic growth of the asset when volatility is zero and therefore we obtain 
\begin{equation*}
u(x,0,t)=e^{-rt}u_T(xe^{rt}).
\end{equation*}
It also satisfies the PDE \eqref{HWPDE2} if $y=0$ and therefore we speak of a \emph{natural boundary condition}. Let us note that one now fixes the non-compatibility of the boundary condition at $x=X$ $u_D(X,y,t)=u_T(X,y)$ with the new boundary condition at $y=0$ by taking the discount factor into account. We stress that the degeneration influences both sub-problems \eqref{FirstSubProb},\eqref{SecondSubProb}. The application of the finite volume method in sub-section \ref{FullDisc} treats the degeneration in the second sub-problem. The boundary corrections, applied to the first sub-problem, have to be computed for $\bar{\alpha}_{i} = \frac{\bar {b}_{i+1/2}}{\bar {a}_{i+1/2}}$ and therefore we set $\bar{\alpha}_{i}$ large enough in order to perform the computations since $\bar {a}_{i+1/2}=0$ if $y=0$. 
\begin{table}[h]
\caption{}
\label{t6}
\centering
\begin{tabular}{cccccc}
\hline
\centering
N & 16$\times$16$\times$32 & 32$\times$32$\times$64 & 64$\times$64$\times$128 & 128$\times$128$\times$256 \\
\hline \hline
$E_{\infty}$ & 2.531 & 1.180 & 0.497 & 0.161 \\
        &  & \emph{(1.102)} & \emph{(1.248)} & \emph{(1.628)} \\
$E_{RMSE}$ & 0.477 & 0.229 & 0.101 & 0.035 \\
        & & \emph{(1.061)} & \emph{(1.178)} & \emph{(1.545)} \\
\hline
\end{tabular}
\end{table}


Convergence results for the original problem $TP1$ for $\zeta=0$ w.r.t. the numerical solution on $256 \times 256 \times 512$ are presented in Table \ref{t6}. We conclude that the numerical method performs well in the case of degeneration in $y$-direction. The $C$ mesh norm error for $\zeta=0$, corresponding to Figures \ref{exactsol},\ref{exactnumsol} on the mesh, sized $32 \times 32 \times 64$, is visualized on Figure \ref{maxerrorexactsol}, while the $C$ norm error for $TP1$, $\zeta=0$ is plot on Figure \ref{maxerrortp1}. 
\begin{figure}[htbp]
\hfill%
\begin{minipage}[b]{0.40\textwidth}
\centering
\includegraphics[width=\textwidth]{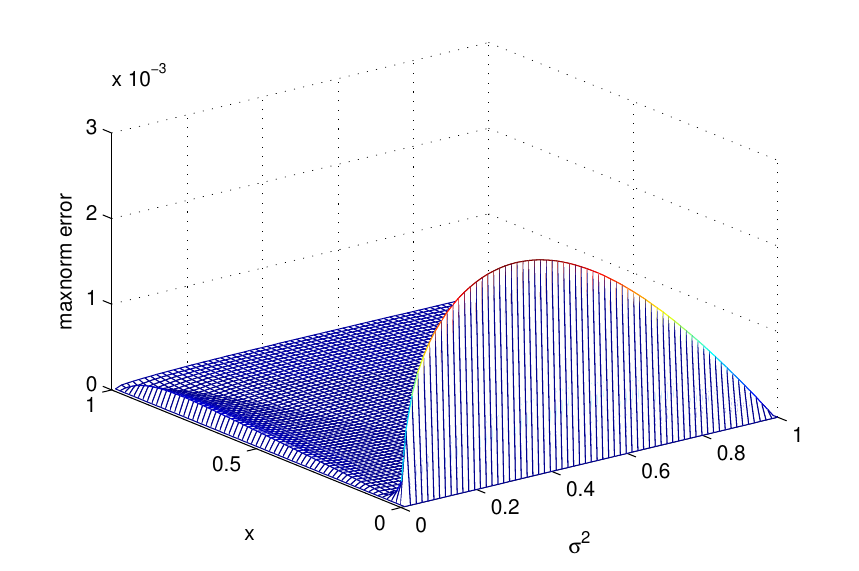}
\caption{maxnorm error exact solution $32 \times 32 \times 64$ $\zeta=0$}\label{maxerrorexactsol}
\end{minipage}%
\hfill%
\begin{minipage}[b]{0.40\textwidth}
\centering
\includegraphics[width=\textwidth]{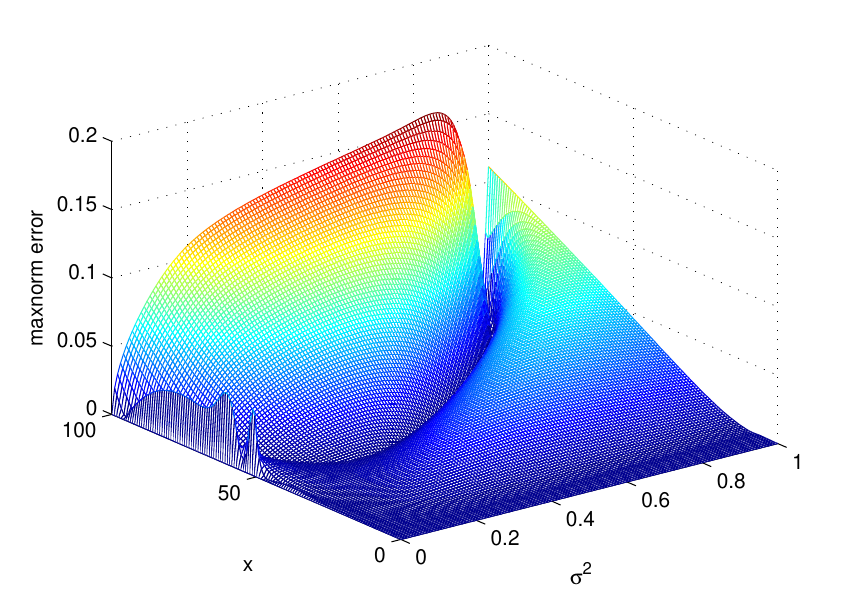}
\caption{maxnorm error $TP1$ $128 \times 128 \times 256$ $\zeta=0$}\label{maxerrortp1}
\end{minipage}\hfill\hbox{}%
\end{figure}

In order to show the effects for the variable stochastic volatility we plot the option values of the 2D and 1D simulations, applied to $TP1$-$TP3$ with and without the stochastic volatility being an independent variable. In the three Figures \ref{comparisontp3}, \ref{comparisontp1}, \ref{comparisontp2}  we see significant differences in those two simulations for fixed values of $\sigma=\sqrt{y}$.
\begin{figure}[htbp]
\hfill%
\begin{minipage}[b]{0.40\textwidth}
\centering
\includegraphics[width=\textwidth]{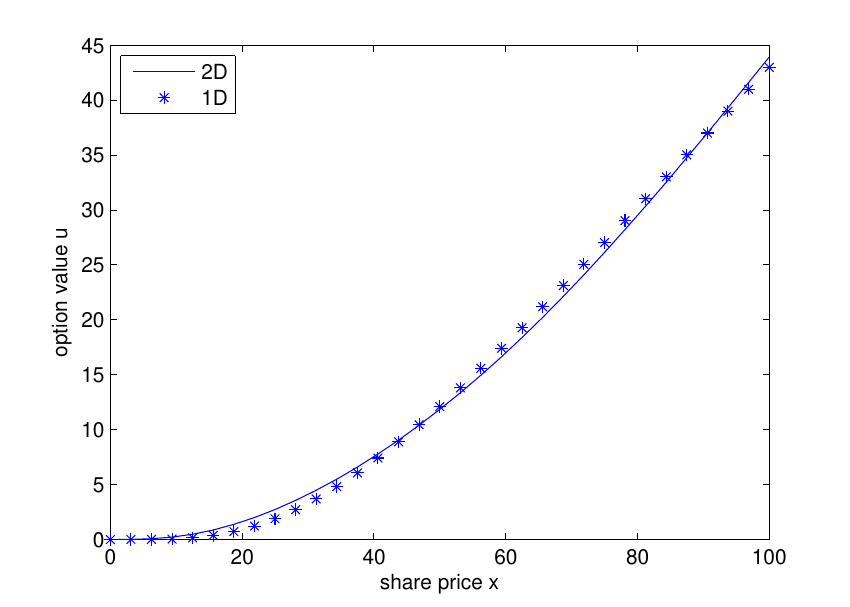}
\caption{2D1D $TP1$ $\sigma \approx 0.71$}\label{comparisontp1}
\end{minipage}%
\hfill%
\begin{minipage}[b]{0.40\textwidth}
\centering
\includegraphics[width=\textwidth]{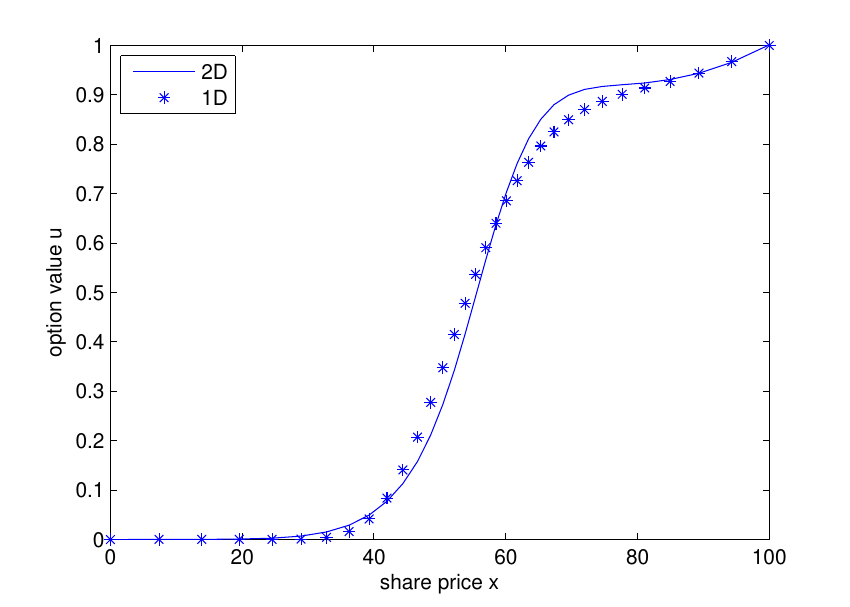}
\caption{2D1D $TP2$ $\sigma \approx 0.18$}\label{comparisontp2}
\end{minipage}\hfill\hbox{}%
\end{figure}

\section{Conclusion}

In this paper we solve numerically the Hull and White 2D problem \eqref{HWPDE}-\eqref{BC} for pricing European options with stochastic volatility. The proposed numerical method consists in LOD operator splitting while in space a fitted finite volume method is applied. We prove first-order convergence in time and present detailed considerations on the discrete maximum principle. The main advantages of the presented scheme are reduction of the computational costs and positivity of the numerical solution in time. Moreover, it produces satisfactory computational results even when degeneration on the boundary $y=0$ is also considered.

In a forthcoming paper we study the stability and the convergence of the proposed splitting finite volume method.

\textbf{Acknowledgement:} The authors would like to thank Prof. Karel in't Hout for his important remarks and suggestions on the differential problem and the numerical method. Also, we are grateful to Dr. Tihomir Gyulov for the helpful discussion on the semi-discrete problem.

This research was supported by the European Union in the FP7-PEOPLE-2012-ITN Program under Grant Agreement Number 304617 (FP7 Marie Curie Action, Project Multi-ITN STRIKE - Novel Methods in Computational Finance) and by the Sofia University Foundation under Grant No 106/2013. The second author is also supported by the Bulgarian National Fund under Project DID 02/37/09.

\end{document}